\documentclass[11pt]{article}
\usepackage{amsmath,amssymb,amsthm,eucal}
\usepackage{color}
\usepackage[all]{xy}
\usepackage{slashed}

\title{\vspace*{-1pc}%
       Shift-tail equivalence and an unbounded representative of the 
        Cuntz-Pimsner extension.}

\author{Magnus Goffeng$^*$, Bram Mesland$^{**}$, Adam Rennie$^\dag$ \\[4pt]
${}^*$Department of Mathematical Sciences,\\ Chalmers University of Technology and the University of Gothenburg,\\ SE-412 96 Gothenburg, Sweden\\[3pt]
${}^{**}$Institut f\"ur Analysis, Leibniz Universit\"at Hannover, Welfengarten 1,\\ 30167 Hannover, Germany
\\[3pt]
${}^\dag$	School of Mathematics and Applied Statistics,
University of Wollongong,\\
Northfields Ave 2522, Australia}

\advance\topmargin by -\headheight
\advance\topmargin by -\headsep
\evensidemargin=\oddsidemargin
\makeatletter
\def\section{\@startsection{section}{1}{\z@}{-3.5ex plus -1ex minus
  -.2ex}{2.3ex plus .2ex}{\large\bf}}
\def\subsection{\@startsection{subsection}{2}{\z@}{-3.25ex plus -1ex
  minus -.2ex}{1.5ex plus .2ex}{\normalsize\bf}}
\makeatother

\numberwithin{equation}{section} 

\theoremstyle{plain} 
\newtheorem{thm}{Theorem}[section]
\newtheorem{lemma}[thm]{Lemma}
\newtheorem{prop}[thm]{Proposition}
\newtheorem{corl}[thm]{Corollary}
\newtheorem{ass}{Assumption}

\theoremstyle{definition} 
\newtheorem{defn}[thm]{Definition}
\newtheorem*{conv}{Convention}

\theoremstyle{remark} 
\newtheorem{rmk}[thm]{Remark}
\newtheorem{ex}[thm]{Example}

\DeclareMathOperator{\diag}{diag} 
\DeclareMathOperator{\Dom}{Dom}   
\DeclareMathOperator{\End}{End}   
\DeclareMathOperator{\Hom}{Hom}   
\DeclareMathOperator{\Id}{Id}     
\DeclareMathOperator{\supp}{supp} 

\newcommand{\extcls}{[\operatorname{ext}]}  

\newcommand{\A}{\mathcal{A}}  
\newcommand{\C}{\mathbb{C}}   
\newcommand{\Cc}{\mathcal{C}} 
\newcommand{\D}{\mathcal{D}}  
\newcommand{\F}{\mathcal{F}}  
\newcommand{\K}{\mathcal{K}}  
\newcommand{\N}{\mathbb{N}}   
\newcommand{\op}{{\rm op}}       
\newcommand{\ox}{\otimes}     
\renewcommand{\O}{\mathcal{O}}  
\newcommand{\T}{\mathcal{T}} 
\newcommand{\Z}{\mathbb{Z}}   

\newcommand{\stroke}{\mathbin|}     

\def\pairL_#1(#2|#3){{}_{#1}(#2\stroke#3)} 
\def\pairR(#1|#2)_#3{(#1\stroke#2)_{#3}} 
\def\scal<#1|#2>{\langle#1\stroke#2\rangle} 

\newbox\ncintdbox \newbox\ncinttbox 
	\setbox0=\hbox{$-$}
	\setbox2=\hbox{$\displaystyle\int$}
	\setbox\ncintdbox=\hbox{\rlap{\hbox
		to \wd2{\hskip-.125em \box2\relax\hfil}}\box0\kern.1em}
	\setbox0=\hbox{$\vcenter{\hrule width 4pt}$}
	\setbox2=\hbox{$\textstyle\int$}
	\setbox\ncinttbox=\hbox{\rlap{\hbox
		to \wd2{\hskip-.175em \box2\relax\hfil}}\box0\kern.1em}

\theoremstyle{plain}
\newtheorem{thm*}{Theorem}

\newcommand{\phimod}{\Xi_A} 
\newcommand{\Fock}{\mathcal{F}_E} 
\newcommand{\algFock}{\mathcal{F}_E^{\textnormal{alg}}} 
\newcommand{\core}{\mathcal{C}_E} 
\newcommand{\topop}{\mathfrak{q}} 

\newcommand{\Dsla}{\slashed{D}}
\newcommand{\aE}{{}_{\alpha}E}
\newcommand{\aXi}{{}_{\alpha}\Xi}

\begin{document}

\maketitle

\vspace{-2pc}

\begin{abstract}
We show how the fine structure in shift-tail equivalence, appearing
in the noncommutative geometry of Cuntz-Krieger algebras developed by the first two listed authors, 
has an analogue in a wide range of other Cuntz-Pimsner
algebras. To illustrate this structure, and where it appears, we produce an unbounded representative
of the defining extension of the Cuntz-Pimsner algebra constructed from a finitely generated projective
bi-Hilbertian bimodule, extending work by the third listed author with
Robertson and Sims. As an application, our construction yields new spectral triples for 
Cuntz- and Cuntz-Krieger algebras and for Cuntz-Pimsner algebras associated to vector bundles 
twisted by an equicontinuous $*$-automorphism.
\end{abstract}

\parskip=5pt
\parindent=0pt

\addtocontents{toc}{\vspace{-1pc}}
\section*{Introduction}
\label{sec:intro}

In this paper we study the noncommutative geometry of Cuntz-Pimsner algebras. 
The end product is an unbounded Kasparov module representing the defining extension 
which reflects the dynamics encoded in the  Cuntz-Pimsner algebra. 
In addition to the important examples of Cuntz-Krieger algebras which arise from sub-shifts of 
finite type, 
Cuntz-Pimsner algebras also include crossed products 
by $\Z$, topological graph algebras and Exel crossed products.

Pimsner's construction \cite{Pimsner} associates to a given bimodule 
$E$ (or $C^*$-correspondence) over a $C^{*}$-algebra $A$ a new $C^{*}$-algebra $\O_{E}$, 
which is to be viewed as the crossed product 
of $A$ by $E$. This viewpoint is in line with the 
idea that an $A$-bimodule $E$ is a generalisation of the notion of $*$-endomorphism, 
and a $*$-endomorphism of a commutative $C^{*}$-algebra corresponds to a 
continuous map of the underlying space. As such, bimodules can be viewed as 
discrete time dynamical systems over $A$. See \cite{deaacaneudoaod} for a 
detailed discussion supporting this point of view.

By construction, the Cuntz-Pimsner algebra $\O_{E}$ associated with a 
finitely generated projective Hilbert-bimodule $E$ 
over a $C^{*}$-algebra $A$ is the quotient in its Toeplitz extension, 
a short exact sequence of $C^{*}$-algebras
\begin{equation}
\label{pimdefext}
0\to \K_A(\F_{E})\to \T_{E} \to \O_{E} \to 0.
\end{equation}
We call the extension  \eqref{pimdefext} the \emph{defining extension} of $\O_E$. 
Here $\T_{E}$ is the algebra of Toeplitz operators on the Fock module $\F_{E}$. 
The $C^*$-algebra $\K_A(\F_{E})$ of $A$-compact operators is Morita equivalent to $A$. 
For $A$ nuclear, the extension \eqref{pimdefext} is semisplit and defines a 
distinguished class $\extcls\in KK^{1}(\O_{E},A)$, see \cite{Kas1}. 
Pimsner showed that the Toeplitz algebra $\T_{E}$ is $KK$-equivalent to $A$, 
and the six term exact sequences relate the $K$-theory and $K$-homology 
of $\O_E$ with that of $A$ through the Pimsner sequence (see \cite[Theorem 4.8]{Pimsner}).

The class of the defining extension $\extcls$ associated to finitely generated projective
bi-Hilbertian modules satisfying an additional technical requirement was represented by
a Kasparov module in \cite{RRS}. The work in \cite{RRS} gives a starting point 
for studying the noncommutative geometry of the corresponding Cuntz-Pimsner algebras, 
i.e. their spectral triples or more generally 
unbounded Kasparov modules \cite{BJ}.

A detailed study of the noncommutative geometry of Cuntz-Krieger algebras $O_{\pmb{A}}$ associated to a 
$\{0,1\}$ matrix $\pmb{A}$ was presented in \cite{GM}. Using
the groupoid model for these algebras, a new ingredient in the form of a function on the groupoid was
introduced and used to construct unbounded Kasparov modules.  
In the present paper we utilise the fact that
Cuntz-Krieger algebras admit a Cuntz-Pimsner model 
over the commutative algebra $C(\Omega_{\pmb{A}})$, where $\Omega_{\pmb{A}}$ is the underlying 
subshift of finite type (see \cite{deaacaneudoaod, W})
to emulate the ideas in \cite{GM} for a wider class of Cuntz-Pimsner algebras. 
In particular, we show that the
noncommutative geometries that were described in \cite{GM},  in fact arise from the 
extension \eqref{pimdefext} associated to this particular model. Thus, a key idea in this paper is to place the construction for 
Cuntz-Krieger algebras in \cite{GM} into the framework for Cuntz-Pimsner algebras of \cite{RRS}.

Algebras that model discrete time dynamical systems quite generally carry a dual circle action. 
The case of bimodule dynamics is no different, 
and every Cuntz-Pimsner algebra carries a canonical circle action inducing a $\Z$-grading. 
While the Pimsner sequence in $KK$-theory relates $KK$-groups of $\O_{E}$ with 
the $KK$-groups of $A$, the Pimsner-Voiculescu sequence in $KK$-theory relates $KK$-groups of $\O_{E}$ with 
those of the core $\Cc_{E}$ (the fixed point algebra for the circle action). 
The literature has mostly focussed on Kasparov cycles associated with the core \cite{careyetal11,gabrielgrensing,paskrennie}.
In fact, if we consider the $\Cc_{E}$ correspondence $E\ox_A\Cc_{E}$ we find that
$\O_{E\ox_A\Cc_{E}}\cong \O_E$ so that
\begin{equation}
\label{pimdefext-core}
0\to \K_{\Cc_{E}}(\F_{E\ox_A\Cc_{E}})\to \T_{E\ox_A\Cc_{E}} \to \O_{E} \to 0
\end{equation}
is exact, \cite{Pimsner}.
The unbounded model for the Kasparov class of the exact sequence \eqref{pimdefext-core}
is well understood 
under mild technical assumptions (the spectral subspace condition, \cite{careyetal11}) 
and arises from the number operator constructed from the dual $\Z$-grading. 
See the discussion in Remark \ref{circleaction} below.
The case of Cuntz-Krieger algebras, and in particular the results of \cite[Section 3.4]{GM}, 
show that it is impossible to describe nontrivial classes in $K^{1}(\O_{\pmb{A}})$ as Kasparov products 
of the extension \eqref{pimdefext-core} with classes in $K^{0}(\Cc_{\pmb{A}})$ by general methods. 
In contrast, in  Theorem \ref{comparingtogm} of the present paper we show that the surjective map 
$K^{0}(C(\Omega_{\bold{A}}))\to K^{1}(O_{\bold{A}})$ constructed in \cite[Theorem 5.2.3, Remark 5.2.6]{GM} 
is in fact the boundary map in the Pimsner sequence arising from the model of $O_{\bold{A}}$ as a 
Cuntz-Pimsner algebra over $C(\Omega_{\bold{A}})$. This shows that the exact sequences associated to the two extensions  \eqref{pimdefext}
and \eqref{pimdefext-core} behave very differently.

Our method employs a refined aspect of the dynamics, arising as part of the data of 
shift-tail equivalence in the Cuntz-Krieger case, and gives further grading information 
needed to assemble an unbounded Kasparov module splitting the extension 
\eqref{pimdefext} and representing the class $\extcls\in KK^1(\O_E,A)$. This
grading is defined on an important module constructed
from the algebra $\O_E$ in \cite{RRS}. The technical novelty of this paper is that of a 
\emph{depth-kore operator}\footnote{Kore is not only phonologically
the same as core, as in gauge fixed point subalgebra, 
but also another name for Persephone -- queen of the Underworld.}. 
The depth-kore operator $\kappa$ detects ``depth" relative to the inductive limit structure of
the core and provides the missing piece when assembling a Dirac operator from 
the number operator associated with the circle action.

The technical setup for this paper is inspired by \cite{KajPinWat}. We let $E$ denote 
a finitely generated projective bi-Hilbertian bimodule
over a unital $C^*$-algebra $A$. That is, $E$ is an $A$-bimodule, 
finitely generated and projective in both actions. Moreover, it is equipped with left and right $A$-valued 
inner products ${}_A(\cdot|\cdot)$, $(\cdot|\cdot)_A$ both making $E$ full, and for which the 
respective actions are faithful and adjointable. We place additional technical assumptions on $E$ 
involving the asymptotic properties of the Jones-Watatani indices of the tensor powers $E^{\otimes_A \ell}$ 
(see Assumption \ref{ass:one} on page \pageref{ass:one} and Assumption \ref{ass:two} on page \pageref{ass:two}). 
These assumptions are satisfied in a large class of examples, with no known  counter-examples at the present time. 
The examples for which the assumptions have been verified include Cuntz-Krieger algebras in the model over the 
one-sided shift space, crossed products by $\Z$ and graph $C^*$-algebras of primitive graphs.

We let $\O_E$ denote the Cuntz-Pimsner algebra constructed from $E$, $\Phi_\infty:\O_E\to A$ the conditional expectation 
constructed in \cite{RRS} assuming Assumption \ref{ass:one} and $\phimod$ the completion of $\O_E$ as 
an $A$-Hilbert module in the conditional expectation $\Phi_\infty$. Under the Assumptions \ref{ass:one} and 
\ref{ass:two} we prove the following theorem. It appears as Theorem \ref{definunbdd} below.

\begin{thm*}
\label{mainthm} 
The $(\O_{E},A)$-bimodule $\Xi_{A}$ decomposes as a direct sum of finitely generated projective $A$-modules
\[
\Xi_{A}=\bigoplus_{n\in \Z}\bigoplus_{k\geq \max\{0,-n\}} \Xi_{A}^{n,k}.
\]
Here $\Xi_{A}^{n,k}=P_{n,k}\Xi_{A}$ for $n\in \Z$, $k\in \N$ with $k+n\geq 0$ and a projection 
$P_{n,k}=P^{*}_{n,k}=P_{n,k}^{2}\in \K_{A}(\Xi_{A})$.
The number operator $c:=\sum_{n,k}nP_{n,k}$ and the depth-kore operator 
$\kappa:=\sum_{n,k}kP_{n,k}$ define self-adjoint regular operators on 
$\phimod$ that commute on the common core given by the algebraic direct sum. 
The operator
\[
\D:=\psi(c,\kappa)=\sum_{n,k}\psi(n,k)P_{n,k},\quad \psi(n,k)
=\left\{\begin{array}{ll} n & k=0\\ -(k+|n|) & \mbox{otherwise}
\end{array}\right. ,
\]
makes $(\O_{E},\Xi_{A},\D)$ into an unbounded Kasparov module representing 
the class of the Toeplitz extension $\extcls\in KK^{1}(\O_{E},A)$.
\end{thm*}

The number operator $c$ and the depth-kore operator $\kappa$ are both canonically 
constructed from $E$. There is however some freedom when choosing $\psi$. This is discussed further in 
Remark \ref{psichange} and Remark \ref{lipprop}.
The cycle in Theorem \ref{mainthm} recovers the well-known number operator construction 
when $E$ is a self-Morita equivalence bimodule, or SMEB (see Proposition \ref{smebcase} in Section \ref{smebsubse}), 
as well as the construction for Cuntz-Krieger algebras in \cite{GM}, 
viewed as Cuntz-Pimsner algebras over the maximal abelian subalgebra 
coming from a subshift of finite type (Section \ref{localhomsubs}). 
Theorem \ref{mainthm} sheds new light on some of the results obtained in \cite{GM}.

An application of Theorem \ref{mainthm} is 
the following construction of a spectral triple for the
Cuntz-Pimsner algebra of a vector 
bundle. Such $C^{*}$-algebras were previously considered in \cite{Dad,Vass}. 
Let $V\to M$ be a complex vector bundle on a Riemannian manifold $M$, $\alpha:C(M)\to C(M)$ 
a $*$-automorphism induced from an isometric $C^1$-diffeomorphism. 
We define ${}_\alpha E:=\Gamma(M,V)$
with the ordinary right $C(M)$-action and the left $C(M)$-action 
defined from $\alpha$ and denote
the associated Cuntz-Pimsner algebra by $\O_{{}_\alpha E}$. 

Consider a Dirac-type operator 
$\Dsla$ on a Clifford bundle $S\to M$ and the Hilbert space
$\mathcal{H}:=\aXi_{C(M)}\otimes_{C(M)}L^2(M,S)$. The operator $\D$ appearing in 
Theorem \ref{mainthm} and the Dirac type operator $\Dsla$ can be assembled to form a 
self-adjoint operator $\D_{E}$ on $\mathcal{H}$. For more details regarding the construction, 
see Section \ref{CPalgofbundle}, in particular Lemma \ref{domainsandshit}. The following 
result appears as Theorem \ref{kasppordvbmand} below.

\begin{thm*}
\label{mainapp} The triple $(\O_{{}_{\alpha}E},\mathcal{H},\D_E)$ 
is a spectral triple for the Cuntz-Pimsner algebra $\O_{{}_{\alpha}E}$ representing the Kasparov 
product of the class of
$$
0\to \K_{C(M)}(\mathcal{F}_{{}_{\alpha}E})\to 
\T_{{}_{\alpha}E}\to 
\O_{{}_{\alpha}E}\to 0
$$
in $KK^1(\O_{{}_{\alpha}E},C(M))$ with $[\Dsla]\in KK^*(C(M),\C)$. 
\end{thm*}

In fact we prove the theorem for `almost isometries', namely $C^1$-diffeomorphisms
inducing automorphisms $\alpha$ such that $\sup_{\ell\in \Z}\Vert[\Dsla,\alpha^\ell(f)]\Vert<\infty$ for
each $f\in C^1(M)$: see Proposition \ref{prop:better-diffeo}. Spectral triples on the crossed 
product $C(M)\rtimes \Z$ of an equicontinuous action, as studied in \cite{BMR}, arise as a special case.

The contents of the paper are as follows. 
In Section \ref{sec:been-done} we recall the setup of \cite{RRS}. 
In particular, we recall the construction of the operator-valued weight 
$\Phi_\infty:\O_E\to A$ used to define the module $\phimod$ of Theorem \ref{mainthm}.  
After recalling the motivating example of Cuntz-Krieger algebras
(from \cite{GM}) in Subsection \ref{briefreviewck} we proceed in 
Section \ref{sec:unbdd} to construct the orthogonal decomposition of 
$\Xi_{A}$ (Section \ref{subsec:chop-chop}), 
the depth-kore operator $\kappa$ and the unbounded cycle (Section  \ref{subsec:kas-mod}) 
appearing in Theorem \ref{mainthm}.

In Section \ref{sec:eggs}, we provide examples in the form of the above mentioned SMEBs and 
Cuntz-Krieger algebras. For the latter we use the construction of the Cuntz-Pimsner algebra 
of a local homeomorphism from \cite{deaacaneudoaod}. 
This clarifies some $K$-theoretic statements proved in \cite{GM}.
In Section \ref{comparingcycletojesus} we compare the dynamical approach 
for the Cuntz algebra $O_N$ with the model using the coefficient algebra $\C$ 
(the graph $C^*$-algebra approach). Finally, in Section \ref{CPalgofbundle} 
we prove Theorem \ref{mainapp}.

{\bf Acknowlegements}
The authors thank the Leibniz Universit\"{a}t Hannover, Germany, 
the University of Wollongong, Australia and 
the Hausdorff Insititute for Mathematics, Bonn, Germany for their hospitality and support. 
The first author was supported by the Knut and Alice Wallenberg foundation and 
the Swedish Research Council Grant 2015-00137 and Marie Sklodowska Curie Actions, 
Cofund, Project INCA 600398.
The second author was supported by EPSRC grant EP/J006580/2. 
The third author was supported by the ARC.
This work has benefited from discussions with Francesca Arici, Robin Deeley, 
Ulrich Kr\"ahmer, Aidan Sims and Michael Whittaker. We thank the anonymous referee for a careful reading of the paper and various helpful suggestions.

\section{The Kasparov module representing the extension class}
\label{sec:been-done}

In this section we will recall the basic setup of \cite{RRS}. 
We have a unital separable $C^*$-algebra $A$, and a
bi-Hilbertian bimodule $E$ over $A$ which is finitely generated and projective for 
both the right and left module structures. This means that $E$ is a bimodule over
$A$, carries both left and right $A$-valued inner products ${}_A(\cdot|\cdot)$, $(\cdot|\cdot)_A$ for
each of which $E$ is full, and for which the respective actions are injective and adjointable. 
See \cite{KajPinWat} for details.
The two inner products automatically yield equivalent norms (see, for instance \cite[Lemma 2.2]{RRS}).
We write ${}_AE$ for $E$ when we wish to emphasize its left module structure and $E_A$ for $E$ 
when emphasizing the right module structure.

\subsection{Cuntz-Pimsner algebras}

Regarding $E$ as a right module with a left $A$-action (a correspondence) we can construct
the Cuntz-Pimsner algebra $\O_E$. This we do concretely in the Fock representation. The
algebraic Fock space is the algebraic direct sum
$$
\algFock=\bigoplus_{\ell\geq 0}^{\textnormal{alg}}E^{\otimes_A \ell}
=\bigoplus_{\ell\geq 0}^{\textnormal{alg}}E^{\otimes \ell}
=A\oplus E\oplus E^{\otimes 2}\oplus\cdots
$$
where the copy of $A$ is the trivial $A$-correspondence. The Fock space $\Fock$ is the completion 
of $\algFock$ as an $A$-Hilbert module. For $\nu\in \algFock$, we define
the creation operator $T_\nu$ by the formula
$$
T_\nu(e_1\ox\cdots\ox e_\ell)=\nu\ox e_1\ox\cdots\ox e_\ell.
$$
The expression $T_\nu$ extends to an adjointable operator on $\Fock$. 
The $C^*$-algebra generated by the set of creation operators $\{T_\nu: \;\nu\in \algFock\}$ 
is the Toeplitz-Pimsner algebra $\T_E$. It is straightforward to show that $\T_E$ contains 
the compact endomorphisms $\K_A(\Fock)$ as an ideal. The defining 
extension for the Cuntz-Pimsner algebra $\O_E$ is the following short exact sequence:
\begin{equation}
0\to \K_A(\Fock)\to \T_E \to \O_E \to 0.
\label{eq:ext}
\end{equation}
For $\nu\in \algFock$, we let $S_\nu$ denote the class of 
$T_\nu$ in $\O_E$. If $\nu\in E^{\otimes \ell}$ we write $|\nu|:=\ell$.
We note that Pimsner's   general construction \cite{Pimsner} uses an ideal that 
in general is larger than $\K_A(\Fock)$. 
In our case, $A$ acts from the left on $E_{A}$ by compact endomorphisms, 
ensuring that Pimsner's ideal coincides with $\K_A(\Fock)$.

\begin{rmk}
The Fock module $\Fock$ is not to be confused with the Fock space defined in the context 
of Cuntz-Krieger algebras and used by Kaminker-Putnam \cite{kaminkerputnam}. 
The constructions in \cite{kaminkerputnam} are related to $KK$-theoretic duality, whereas 
our aim is to represent a specific extension class by an unbounded cycle.
\end{rmk}

\begin{rmk}
\label{circleaction}
The formula $z\cdot S_\nu S_\mu^*:=z^{|\nu|-|\mu|}S_\nu S_\mu^*$ extends to a $U(1)$-action on $\O_E$, \cite{Pimsner}. 
We denote the fixed point algebra for this action by $\core$. The formula 
$\rho(x):=\int_{U(1)} z\cdot x\, \mathrm{d}z$ (where $\mathrm{d}z$ 
denotes the normalized Haar measure on $U(1)$) 
defines a conditional expectation $\rho:\O_E\to \core$. The generator of the circle action 
defines a closed operator $N$ on the completion $X_{\core}$ of $\O_E$ as a $\Cc_{E}$-Hilbert module
in the inner product defined from $\rho$. Under the spectral subspace assumption 
(see \cite[Definition 2.2]{careyetal11}), $N$ is a self-adjoint, regular operator with compact resolvent 
whose commutators with $\{S_\nu: \;\nu\in \algFock\}$ are bounded. In particular, 
$(\O_E,X_{\core},N)$ 
defines an unbounded $(\O_E,\core)$-Kasparov module.

There is an equality $\Cc_{E}=A$ if and only if $E$ can be given a left inner product making it
a self-Morita equivalence bimodule (SMEB), \cite[Proposition 5.18]{K2}. 
SMEB's are considered further in Example \ref{smebexample}.
This case has been studied in \cite{RRS} as well as in \cite{gabrielgrensing}. 
In general, $\Cc_{E}$ is substantially larger than $A$ and the generator of the circle action is 
insufficient for constructing an unbounded $(\O_E,A)$-Kasparov module. 
\end{rmk}

\begin{ex}[Local homeomorphisms]
\label{localexamnple}
Let $g:V\to V$ be a local homeomorphism of a compact space $V$. Associated with $g$, there is a 
transfer operator 
$$
\mathfrak{L}:C(V)\to C(V), \quad \mathfrak{L}(f)(x):=\sum_{g(y)=x}f(y).
$$
We can define a bimodule 
structure $E={}_{\mathrm{id}}C(V)_{g^*}$ on $C(V)$ by 
$$(afb)(x)=a(x)f(x)b(g(x)), \quad a,b\in C(V), \;f\in E.$$
The two inner products on $E$ are given by
$$
(f_1|f_2)_{C(V)}:=\mathfrak{L}(\overline{f_1}f_2)\quad\mbox{and}
\quad {}_{C(V)}(f_1|f_2)=f_1\overline{f_2}.
$$
For more details see \cite{deaacaneudoaod}. As a source of examples, we will mainly be concerned with a special case: the shift mapping on a subshift of finite type. 
The reason for this is that the associated Cuntz-Pimsner algebra is a Cuntz-Krieger algebra, and as such 
it also admits a model as a Cuntz-Pimsner algebra over a finite-dimensional $C^*$-algebra. 
This will allow us to compare and contrast our techniques relative to the choice of Cuntz-Pimsner 
model rather than the isomorphism class of the Cuntz-Pimsner algebra.
\end{ex}

\begin{ex}[Graph $C^*$-algebras]
Let $G=(G^0,G^1,r,s)$ be a finite directed graph. We consider the finite-dimensional algebra $A=C(G^0)$ and 
the $A$-bimodule $E=C(G^1)$ with the bimodule structure
$$(afb)(g)=a(r(g))f(g)b(s(g)), \quad a,b\in A, \; f\in E.$$
For $e,f\in E$, the inner products are defined by 
$$(e|f)_A(v):=\sum_{s(g)=v} \overline{e(g)}f(g)\quad \mbox{and}\quad {}_A(e|f)(v):=\sum_{r(g)=v} e(g)\overline{f(g)}.$$
The associated Cuntz-Pimsner algebra coincides with the graph $C^*$-algebra $C^*(G)$ , see \cite[Example 2, p 193]{Pimsner}.
\end{ex}

\begin{ex}[Cuntz-Krieger algebras]
Assume that $\pmb{A}:=(a_{ij})_{i,j=1}^N$ is an $N\times N$-matrix of $0$'s and $1$'s. We let $O_{\pmb{A}}$ 
denote the associated Cuntz-Krieger algebra \cite{CK}. If $\pmb{A}$ is the edge adjacency matrix of a finite directed graph 
$G$, then $O_{\pmb{A}}\cong C^*(G)$. On the other hand, letting $(\Omega_{\pmb{A}},\sigma)$ denote the 
associated one-sided subshift of finite type $O_{\pmb{A}}$ coincides with the Cuntz-Pimsner algebra associated 
with the local homeomorphism $\sigma$ as in Example \ref{localexamnple}. Yet another description is in terms 
of groupoids; $O_{\pmb{A}}$ is isomorphic to the groupoid $C^*$-algebra of the groupoid 
\begin{equation}
\label{thegroupoidra}
\mathcal{R}_{\pmb{A}}:=\{(x,n,y)\in \Omega_{\pmb{A}}\times \Z\times \Omega_{\pmb{A}}: 
\;\exists k\geq \max\{0,-n\} \mbox{  with  } \sigma^{n+k}(x)=\sigma^k(y)\}\rightrightarrows\Omega_{\pmb{A}}.
\end{equation}
That is,  $\mathcal{R}_{\pmb{A}}$ consists of shift-tail equivalent pairs of points with a prescribed lag. 
The set $\mathcal{R}_{\pmb{A}}$ becomes a groupoid for the operation $(x,n,y)(y,m,z)=(x,n+m,z)$ and 
can be equipped with an \'etale topology (see \cite{Ren1,Ren2}). 
\end{ex}

A Cuntz-Krieger algebra $O_{\pmb{A}}$ also has a graph $C^*$-algebra model. 
However, we use the convention that whenever referring to a Cuntz-Krieger algebra as a Cuntz-Pimsner algebra 
we mean its model over $C(\Omega_{\pmb{A}})$. We distinguish it from its Cuntz-Pimsner model as a graph $C^*$-algebra.

\subsection{The conditional expectation}

A Kasparov module representing the class of the extension \eqref{eq:ext}
was constructed in \cite{RRS}. Here we recall the salient points. 
For $x$ and $y$ in a right Hilbert module, we denote the associated rank-one 
operator by $\Theta_{x,y}:=x\langle y,\cdot\rangle$. 
We choose a frame $(e_\rho)_{\rho=1}^N$ for $E_A$. By a frame we mean
$$
\sum_{\rho=1}^N\Theta_{e_\rho,e_\rho}={\rm Id}_E.
$$
The frame $(e_\rho)_{\rho=1}^N$ induces a frame for $E_A^{\ox \ell}$, namely
$(e_\rho)_{|\rho|=\ell}$ where $\rho$ is a multi-index and $e_\rho=e_{\rho_1}\ox\cdots\ox e_{\rho_\ell}$.

We use ideas from \cite{KajPinWat} to define an $A$-bilinear functional $\Phi_\infty:\O_E\to A$. 
The details of this construction can be found in \cite[Section 3.2]{RRS}.
This functional will furnish us with an $A$-valued inner product on $\O_E$. We define 
$$
\Phi_\ell:\End_A^*(E^{\ox \ell})\to A,
\qquad \Phi_\ell(T)=\sum_{|\rho|=\ell}{}_A(Te_\rho|e_\rho).
$$
Here we use the notation $\End_A^*(E^{\ox \ell})$ for the $C^*$-algebra of 
$A$-linear adjointable operators on $E^{\ox \ell}$. It follows from \cite[Lemma 2.16]{KajPinWat}
that $\Phi_\ell$ does not depend on the choice of frame. We write 
$\mathrm{e}^{\beta_\ell}:=\Phi_\ell({\rm Id}_{E^{\ox \ell}})$. Since $\Phi_\ell$ is independent of the
choice of frame, so is $\mathrm{e}^{\beta_\ell}$. Note that $\mathrm{e}^{\beta_\ell}$ is a 
positive, central, invertible element of $A$. Therefore $\beta_\ell$ is a well defined 
self-adjoint central element in $A$. We extend the functional $\Phi_\ell$
to a mapping $\End^*_A(\Fock)\to A$ by compressing along the orthogonally complemented submodule 
$E^{\ox \ell}\subseteq \Fock$. 

Naively, we would like to define
\begin{equation}
\label{phiinftydef}
\Phi_\infty(T)\,``\!:=\!"\lim_{\ell\to \infty}\Phi_\ell(T)\mathrm{e}^{-\beta_\ell},
\quad\mbox{for suitable $T\in \End_A^*(\Fock)$.}
\end{equation}
Indeed, $\Phi_\ell(T)\mathrm{e}^{-\beta_\ell}$ is easily shown to be bounded and some `generalised limit'
might exist. We can not employ the theory of generalised limits as $\Phi_\infty$ is not 
scalar valued. Following \cite{RRS}, we work under the following 
assumption guaranteeing that the limit exists for $T$ in a dense subspace of $\T_E$.

\begin{ass}
\label{ass:one}
We assume that for every $\ell\in \N$ and $\nu\in E^{\otimes \ell}$, there is a $\delta>0$ and a $\tilde{\nu}\in E^{\otimes \ell}$ 
such that 
$$
\Vert \mathrm{e}^{-\beta_n}\nu \mathrm{e}^{\beta_{n-\ell}}-\tilde\nu\Vert = O(n^{-\delta}),\quad\mbox{as $n\to \infty$}.
$$
\end{ass}

\begin{ex}[Assumption \ref{ass:one} and graph $C^*$-algebras]
\label{discongrap}
Assumption \ref{ass:one} is non-trivial for graph $C^*$-algebras. It was verified in \cite[Example 3.8]{RRS} that 
a graph $C^*$-algebra with primitive vertex adjacency matrix satisfies Assumption \ref{ass:one}. 
The Jones-Watatani indices of a graph $C^*$-algebra were computed in \cite[Example 3.8]{RRS} 
by means of its vertex adjacency matrix $\pmb{A}_{\pmb{v}}$ as 
$$\mathrm{e}^{\beta_\ell}=\sum_{v,w\in G^0} \pmb{A}_{\pmb{v}}^\ell(v,w)\delta_v\in A=C(G^0).$$
If $G$ is the graph with $N$ edges on one vertex, $C^*(G)=O_N$ and $\mathrm{e}^{\beta_\ell}=N^\ell$. 
It is an open problem to determine if all graph $C^*$-algebras satisfy Assumption \ref{ass:one}.
\end{ex}

\begin{ex}[Assumption \ref{ass:one} for Cuntz-Krieger algebras]
\label{betakcomp}
Let us verify Assumption \ref{ass:one} for Cuntz-Krieger algebras in the Cuntz-Pimsner model 
over $C(\Omega_{\pmb{A}})$ for an $N\times N$-matrix $\pmb{A}$. 
We will choose a frame for $E={}_{\mathrm{id}}C(\Omega_{\pmb{A}}))_{\sigma^*}$ 
as follows. A left frame is given by the constant function $1\in E$. 
To construct a right frame, choose a covering $(U_j)_{j=1}^M$ such that $\sigma|:U_j\to \sigma(U_j)$ is a 
homeomorphism. We also pick a subordinate partition of unity $(\chi_j^2)_{j=1}^M$ 
(i.e. $\supp(\chi_j)\subseteq U_j$ and $\sum_j \chi_j^2=1$). For instance, with $M=N$ the cylinder sets 
$$U_j:=\{x=x_0x_1x_2\cdots \in \Omega_{\pmb{A}}: \;x_0=j\},$$
will form a clopen cover and $\chi_j^2:=\chi_{U_j}$ is a subordinate partition of unity.
We claim that $e_j:=\chi_j$ defines a right frame. Indeed, for any $f\in E$ the following identity holds
\begin{align*}
\left[\sum_j \chi_j (\chi_j|f)_{C(\Omega_{\pmb{A}})}\right](x)&=\sum_j\chi_j(x)\mathfrak{L}(\chi_j f)(\sigma(x))
=\sum_j\sum_{\sigma(y)=\sigma(x)}\chi_j(x)\chi_j(y)f(y)\\
&=\sum_j \chi_j^2(x)f(x)=f(x),
\end{align*}
where in the second last step we used the fact that $\sigma$ is injective on $U_j$. 
For a multi index $\rho$ of length $r$, we use the notation
$$\chi_\rho(x):=\prod_{j=1}^r \chi_{\rho_j}(\sigma^{j-1}(x)).$$

A simple computation gives 
\begin{equation*}
\textrm{e}^{\beta_\ell}=\sum_{|\rho|=\ell} \chi_\rho^2=\prod_{j=1}^\ell\left(\sum_{k=1}^M \chi_k^2(\sigma^{j-1}(x))\right)=1.
\end{equation*}
Therefore $\beta_\ell=0$ for Cuntz-Krieger algebras and Assumption \ref{ass:one} is satisfied. We remark at this point 
that the Jones-Watatani index is associated to the module and not the Cuntz-Pimsner algebra constructed from the module.
For a Cuntz-Krieger algebra we see enormous differences between the model over $C(\Omega_{\pmb{A}})$ 
and the model as a graph $C^*$-algebra.
\end{ex}

{\bf We assume that Assumption \ref{ass:one} holds for the remainder of the paper.}

In \cite{RRS}, the reader can find further examples of Cuntz-Pimsner algebras for 
which Assumption \ref{ass:one} holds. There are no known examples for which 
Assumption \ref{ass:one} does not hold.
When Assumption \ref{ass:one} holds, \cite[Proposition 3.5]{RRS} guarantees that the expression 
in Equation \eqref{phiinftydef} is well-defined on the $*$-algebra generated by the set of 
creation operators $\{T_\nu: \;\nu\in \algFock\}$. Indeed, we can under Assumption \ref{ass:one} 
compute $\Phi_\infty$ on the $*$-algebra generated by $\{T_\nu: \;\nu\in \algFock\}$.

\begin{lemma}
\label{computeinnerprod} 
For homogeneous elements $\mu,\,\nu\in \algFock$ we have
\begin{equation}
\label{eq:mu-nu-beta}
\Phi_\infty(T_\mu T_\nu^*)
= \lim_{\ell\to\infty}{}_A(\mu|\mathrm{e}^{-\beta_\ell}\nu \mathrm{e}^{\beta_{\ell-|\nu|}}).
\end{equation}
In particular, if $T$ is homogeneous of degree $n$, $|n|>0$, then $\Phi_{\infty}(T)=0$.
\end{lemma}

\begin{proof} 
It is proved in \cite[Lemma 3.2]{RRS} that for homogeneous $\mu,\,\nu\in \algFock$ we have 
$$
\Phi_\ell(T_\mu T_\nu^*)={}_A(\mu|\nu \mathrm{e}^{\beta_{\ell-|\nu|}})
$$
whenever $\ell\geq |\mu|=|\nu|$. Therefore, assuming Assumption \ref{ass:one} we have
\begin{align}
\Phi_\infty(T_\mu T_\nu^*)
&=\lim_{\ell\to \infty}\Phi_k(T_\mu T_\nu^*)\mathrm{e}^{-\beta_\ell}
=\lim_{\ell\to \infty}{}_A(\mu|\nu \mathrm{e}^{\beta_{\ell-|\nu|}})\mathrm{e}^{-\beta_\ell}=\lim_{\ell\to\infty}{}_A(\mu|\mathrm{e}^{-\beta_\ell}\nu \mathrm{e}^{\beta_{\ell-|\nu|}}).
\nonumber
\end{align}
Now if $T$ is of degree $n$, $|n|>0$, then $T$ is a 
linear combination of elements of the form $T_{\mu}T^{*}_{\nu}$ 
with $|\mu|\neq |\nu|$ and therefore 
${}_A(\mu| \mathrm{e}^{-\beta_{\ell}}\nu \mathrm{e}^{\beta_{\ell-|\nu|}})=0$ 
for all $\ell$, giving the desired statement.
\end{proof}

By a positivity argument, the mapping $\Phi_\infty$ is continuous in the $C^*$-norm on 
the $*$-algebra generated by $\{T_\nu: \;\nu\in \algFock\}$. 
We extend by continuity to obtain a unital positive $A$-bilinear functional $\Phi_\infty: \T_E\to A$.
The functional $\Phi_\infty$ annihilates the compact endomorphisms, and descends to a well-defined
functional on the Cuntz-Pimsner algebra $\O_E$. By an abuse of notation, we 
denote also this functional by $\Phi_\infty:\O_E\to A$. The reader is referred to \cite{RRS}, and in particular 
\cite[Proposition 3.5]{RRS}, for further details on $\Phi_\infty$. 
Since $\Phi_\ell$ and $\mathrm{e}^{\beta_\ell}$ 
do not depend on the choice of frame, neither does $\Phi_\infty$. 

In examples, the conditional expectation is computable. For instance, it was proven in \cite[Example 3.6]{RRS} that 
for the graph $G$ with $N$ edges on one vertex, so $A=\C$ and $E=\C^N$, with $C^*(G)=O_E$ being the Cuntz algebra 
$O_N$, then $\Phi_\infty$ coincides with the unique KMS-state for the gauge action on $O_N$, so
\begin{equation}\label{KMSCuntz}
\Phi_\infty(S_\mu S_\nu^*)=\delta_{\mu,\nu}N^{-|\mu|}.
\end{equation}
For a Cuntz-Krieger algebra we can also compute $\Phi_\infty$. 

\begin{conv}
\label{conventionforsimple}
Given a simple tensor $\nu\in \algFock$, with $\nu=\nu_1\ox\nu_2\ox\cdots\ox\nu_\ell$,
we will write $\nu=\underline{\nu}\overline{\nu}$ with $\underline{\nu}=\nu_1\ox\cdots\ox\nu_m$ and 
$\overline{\nu}=\nu_{m+1}\ox\cdots\ox \nu_\ell$ where
$m\leq \ell$ will either be clear from context or specified. 
\end{conv}

\begin{lemma}
\label{psiinftyandunit}
Let $\pmb{A}$ denote an $N\times N$-matrix of $0$'s and $1$'s, $\mathcal{R}_{\pmb{A}}$ the associated 
groupoid as in Equation \eqref{thegroupoidra} and $\Phi_\infty:C^*(\mathcal{R}_{\pmb{A}})\to C(\Omega_{\pmb{A}})$ 
the conditional expectation associated with the Cuntz-Pimsner model. For $f\in C_c(\mathcal{R}_{\pmb{A}})$ it holds that 
$$
\Phi_\infty(f)(x)=f(x,0,x).
$$
\end{lemma}

\begin{proof}
It suffices to prove that for $f_{\mu,\nu}\in C_c(\mathcal{R}_{\pmb{A}})$ defined from an element 
$S_\mu S_\nu^*$, where $\mu,\nu\in E^{\ox m}$, the identity $\Phi_\ell(S_\mu S_\nu^*)(x)=f_{\mu,\nu}(x,0,x)$ 
holds whenever $\ell>m$. We note that for general homogeneous 
$\mu,\nu\in \mathcal{F}^{\textnormal{alg}}_E$:
$$f_{\mu,\nu}(x,n,y)=
\begin{cases}
\mu(x)\nu^{*}(y) \;&\mbox{if}\; n=|\mu|-|\nu|\;\mbox{and} \; \sigma^{|\mu|}(x)=\sigma^{|\nu|}(y),\\
0 \;&\mbox{otherwise}. \end{cases}
$$
Here we are using the fact that $E^{\ox m}\cong C(\Omega_{\pmb{A}})$ as linear spaces
for any $m$ to identify $\mu$ and $\nu$ with 
functions. We denote the conjugate function by $\nu^{*}$ to avoid notational ambiguity later. 
Let $(e_j)_{j=1}^M$ denote the frame from Example \ref{betakcomp}, associated with a partition of unity subordinate to the 
cover $(U_j)_{j=1}^M$. Note that for $\rho=(\rho_1,\rho_2,\dots,\rho_\ell)$, we identify $e_\rho$ with the function
$$
e_\rho(x):=\chi_\rho(x).
$$
 We write
$$
\Phi_\ell(S_\mu S_\nu^*)=
\sum_{|\rho|=\ell} {}_{C(\Omega_{\pmb{A}})}(\mu\otimes (\nu|e_{\underline{\rho}})_{C(\Omega_{\pmb{A}})} e_{\overline{\rho}}|e_\rho),$$
where $|\underline{\rho}|=m=|\nu|$. After some short computations, we see that
\begin{align*}
\left[\sum_{|\rho|=\ell} {}_{C(\Omega_{\pmb{A}})}(\mu\otimes (\nu|e_{\underline{\rho}})_{C(\Omega_{\pmb{A}})} e_{\overline{\rho}}|e_\rho)\right](x)
&=\sum_{|\rho|=\ell}\mu(x)\mathfrak{L}^m[\nu^{*}\chi_{\underline{\rho}}](\sigma^m(x))\chi_{\overline{\rho}}(\sigma^m(x))\chi_{\rho}(x)\\
&=\sum_{|\rho|=\ell}\mu(x)\mathfrak{L}^m[\nu^{*}\chi_{\underline{\rho}}](\sigma^m(x))
\chi_{\overline{\rho}}^2(\sigma^m(x))\chi_{\underline{\rho}}(x)\\
&=\sum_{|\rho|=\ell}\sum_{\sigma^m(x)=\sigma^m(y)}\mu(x)\nu^{*}(y)\chi_{\underline{\rho}}(y)
\chi_{\overline{\rho}}^2(\sigma^m(x))\chi_{\underline{\rho}}(x)\\
&=\sum_{|\rho|=\ell}\mu(x)\nu^{*}(x)\chi_{\rho}^2(x)=
\mu(x)\nu^{*}(x)=f_{\mu,\nu}(x,0,x).
\end{align*}
We used the injectivity of $\sigma$ on $U_j$ in the third equality. 
\end{proof}

\subsection{A bounded Kasparov module for $[\mathrm{ext}]$}

We equip $\O_{E}$ with the $A$-valued inner product
$$
(S_1|S_2)_A:=\Phi_\infty(S_1^*S_2),\qquad S_1,\,S_2\in \O_E.
$$
Completing $\O_E$ (modulo the vectors of zero length) with respect to $\Phi_\infty$ 
yields a right $A$-Hilbert $C^*$-module 
that we denote by $\phimod$. The module $\phimod$ carries a left action of $\O_E$ given by 
extending the multiplication action of $\O_E$ on itself.

\begin{ex}
For a Cuntz-Krieger algebra defined from the $N\times N$-matrix $\pmb{A}$, 
$\Xi_{C(\Omega_{\pmb{A}})}$ coincides with the left regular 
representation $L^2(\mathcal{R}_{\pmb{A}})_{C(\Omega_{\pmb{A}})}$ of the groupoid 
model $O_{\pmb{A}}\cong C^*(\mathcal{R}_{\pmb{A}})$ by Lemma \ref{psiinftyandunit}. 
\end{ex}

By considering the linear span of the image of the generators $S_\nu$, $\nu\in \algFock$, inside
the module $\phimod$, we obtain an isometrically embedded copy of the Fock space $\Fock$.
This fact follows from the identity 
$$(S_\mu|S_\nu)_A=\Phi_\infty(S_\mu^*S_\nu)=( \mu|\nu)_A,$$
using that $S_\mu^*S_\nu=( \mu|\nu)_A$ in $O_E$.
We let $Q$ be the projection on this copy of the Fock space.

\begin{thm}[Proposition 3.14 of \cite{RRS}]
The tuple $(\O_E,\phimod,2Q-1)$ is an odd Kasparov module representing
the class of the extension $[\mathrm{ext}]$ defined in \eqref{eq:ext}.
\end{thm}

\begin{ex}
\label{smebexample}
A self Morita equivalence bimodule (SMEB) $E$ is a bimodule over $A$ as above 
whose left and right inner products satisfy the compatibility condition
$$
\mu(\xi|\eta)_A={}_A(\mu|\xi)\eta,\quad\forall \mu,\xi,\eta\in E.
$$
Equivalently, $E$ is equipped with a right inner product and there is an isomorphism 
$A\cong \K_A(E)$ defining the left inner product. In particular, $E$ defines a 
Morita equivalence $A\sim_M A$. 
When $E$ is a SMEB, $\Phi_{\infty}:\O_{E}\to A$ coincides 
with the expectation $\rho:\O_{E}\to \Cc_{E}$ discussed in Remark \ref{circleaction}. 
Therefore
$$
\phimod=\bigoplus_{n\in\Z}E^{\ox n}
$$
where $E^{\ox(-|n|)}=\overline{E}^{\ox |n|}$. In general the module 
$\phimod$ is more complicated.
This fact will be captured by the depth-kore operator 
$\kappa$ (see below in Subsection \ref{subsec:kas-mod}). 
\end{ex}

For $\mu,\,\nu\in \algFock$, we denote the image of the generator $S_\mu S_\nu^*\in \O_E$
in the module $\Xi_{A}$ by $[S_\mu S_\nu^*]=W_{\mu,\nu}$. 
We also use the notation $W_{\mu,\emptyset}:=[S_\mu]$ and $W_{\emptyset, \nu}:=[S_\nu^*]$.
Denote by $\Xi_{A}^{0}$ the completion 
of the fixed point algebra $\Cc_{E}$ in the inner product defined by the restriction of $\Phi_{\infty}$.
For $n\in \Z$, we let $\Xi_A^n$ denote the closed linear span of $\{W_{\mu,\nu}: |\mu|-|\nu|=n\}$ 
inside $\Xi_A$.

\begin{lemma} 
\label{wherewefindpsin}
Recall the unbounded Kasparov module $(\O_E,X_{\core},N)$ from Remark \ref{circleaction}.
The right $A$-module $\Xi_{A}$ decomposes as a tensor product 
$$
\Xi_{A}\cong X_{\core}\otimes_{\Cc_{E}}\Xi^{0}_{A}.
$$ 
Consequently, for $z\in U(1)$ the prescription $U_zW_{\mu,\nu}:=z^{|\mu|-|\nu|}W_{\mu,\nu}$ 
defines a $U(1)$-action on $\Xi_{A}$. The associated projections 
$\Psi_{n}:\Xi_{A}\to \Xi_{A}^{n}$ onto the spectral subspaces are adjointable and 
there is a direct sum decomposition
$$
\Xi_{A}\cong \bigoplus_{n\in\Z} \Xi_{A}^{n}.
$$
\end{lemma}

\begin{proof} Because the multiplication map $\O_{E}\otimes^{\textnormal{alg}}\Cc_{E}\to \O_{E}$ 
has dense range in $\O_{E}$, we only have to verify that the inner products coincide under this map. 
This follows by computing, for $|\mu_{i}|=|\nu_{i}|, i=1,2$:
\begin{align*} 
\langle S_{\alpha_{1}}S^{*}_{\beta_{1}}\otimes S_{\mu_{1}}S_{\nu_{1}}^{*}, 
S_{\alpha_{2}}S^{*}_{\beta_{2}}\otimes S_{\mu_{2}}S_{\nu_{2}}^{*} 
\rangle_{\O_{E}^{\rho}\otimes_{\Cc_{E}}\Xi^{0}_{E}} 
&= \Phi_{\infty}(S_{\nu_{1}}S_{\mu_{1}}^{*}\rho(S_{\beta_{1}}S^{*}_{\alpha_{1}}S_{\alpha_{2}}S^{*}_{\beta_{2}})S_{\mu_{2}}S_{\nu_{2}}^{*})\\
&=\delta_{|\alpha_{1}|-|\beta_{1}|,|\alpha_{2}|-|\beta_{2}|}
\Phi_{\infty} (S_{\nu_{1}}S_{\mu_{1}}^{*}S_{\beta_{1}}S^{*}_{\alpha_{1}}S_{\alpha_{2}}S^{*}_{\beta_{2}}S_{\mu_{2}}S_{\nu_{2}}^{*})\\
&=(S_{\alpha_{1}}S^{*}_{\beta_{1}}S_{\mu_{1}}S_{\nu_{1}}^{*}| S_{\alpha_{2}}S^{*}_{\beta_{2}}S_{\mu_{2}}S_{\nu_{2}}^{*})_{A}, \end{align*}
by Lemma \ref{computeinnerprod}. The statements on the $U(1)$-action 
and adjointability of the projections $\Psi_{n}$ now follow immediately.
\end{proof}

\section{An unbounded representative of the extension class}
\label{sec:unbdd}

In this section we will use ideas from \cite{GM} to define an unbounded operator on the
module $\phimod$. The issues of self-adjointness and regularity will
be rendered trivial by defining our operator in diagonal form. This relies
on having an orthogonal decomposition of our module into finitely generated projective submodules.

\subsection{Brief review of the construction for Cuntz-Krieger algebras}
\label{briefreviewck}
Before going into the general construction, let us briefly recall how the orthogonal decomposition 
into finitely generated projective submodules is constructed for Cuntz-Krieger algebras. This example 
is explained in detail in \cite{GM}. The precise relation to the general construction appearing below 
in Section \ref{subsec:chop-chop} can be found in Subsection \ref{subsec:dec-exam}. 
The Cuntz-Krieger algebras are Cuntz-Pimsner algebras, but the structure in which said 
decomposition becomes more transparent is in the picture using the shift-tail equivalence 
groupoid $\mathcal{R}_{\pmb{A}}$, defined in Equation \eqref{thegroupoidra}.
To decompose the left regular representation $L^{2}(\mathcal{R}_{\pmb{A}})_{C(\Omega_{\pmb{A}})}$ 
into finitely generated projective submodules, both of the parameters $n$ and $k$ have to be 
taken into account. For $k\geq \max\{0,-n\}$, we define the compact sets
\begin{equation}
\label{thegroupoidradecomp}
\mathcal{R}_{\pmb{A}}^{n,k}:=\left\{(x,n,y)\in \mathcal{R}_{\pmb{A}}: 
\sigma^{n+k}(x)=\sigma^k(y)\;\mbox{ and }
k= \max\{0,-n\}\;\mbox{ or }\sigma^{n+k-1}(x)\neq\sigma^{k-1}(y) \; \right\}.
\end{equation}
The modules $C(\mathcal{R}_{\pmb{A}}^{n,k})$ are finitely generated projective 
$C(\Omega_{\pmb{A}})$-modules, and $\oplus_{n,k}C(\mathcal{R}_{\pmb{A}}^{n,k})
\subseteq C_c(\mathcal{R}_{\pmb{A}})$ 
gives a dense $C(\Omega_{\pmb{A}})$-sub-module of $O_{\pmb{A}}$. 
These modules are orthogonal for the canonical $C(\Omega_{\pmb{A}})$-valued inner product on 
$O_{\pmb{A}}$ (for support reasons). The depth-kore operator $\kappa$ we seek should mimic 
the multiplication operator by the function $\kappa_{\pmb{A}}\in C(\mathcal{R}_{\pmb{A}})$ defined by 
\begin{equation}
\label{kappaadef}
\kappa_{\pmb{A}} (x,n,y):=\min \big\{k\geq \max\{0,-n\}: \sigma^{n+k}(x)=\sigma^k(y)\big\}.
\end{equation}
The function $\kappa_{\pmb{A}}$ supplements the cocycle $c$ defined by $c(x,n,y)=n$
to provide the orthogonal decomposition of $L^{2}(\mathcal{R}_{\pmb{A}})_{C(\Omega_{\pmb{A}})}$ 
via $\mathcal{R}_{\pmb{A}}^{n,k}=c^{-1}(n)\cap \kappa_{\pmb{A}}^{-1}(k)$. We now turn to the case of  more general Cuntz-Pimsner algebras and return to this example in Subsection \ref{subsec:dec-exam}.

\subsection{An orthogonal decomposition}
\label{subsec:chop-chop}

To construct a self-adjoint regular operator, we first analyse the structure of the module $\phimod$. More precisely, 
we will construct a densely defined operator $\kappa$ that tames the wild 
structure of $\Fock^{\perp}\subset\Xi_{A}$. 

\begin{rmk}
Further motivation for the construction below can be found when comparing to the SMEB case 
(cf. \cite[Theorem 3.1]{RRS} and Example \ref{smebexample}). 
The negative spectral subspace of $\phimod$ for the number operator in the SMEB case
is given by the direct sum of all powers of $\overline{E}$. 
The right module structure on $\overline{E}$ comes
from the left module structure on $E$. For a SMEB, the change of 
module structure from $E^{\ox \ell}$ to $\overline{E}^{\otimes \ell}$ 
is harmless, as the left and right
module structures are closely related.
In the general case, the two module structures are in principle
(and in practise) quite different. Therefore, when mapping powers of $\overline{E}$ into
$\phimod$ by $\bar{e}\to S_e^*$ orthogonality is not preserved and no isometric property holds.
\end{rmk}

To construct $\kappa$, we will add an additional assumption
regarding the fine structure of the operation $\nu\mapsto \tilde{\nu}$ in 
Assumption \ref{ass:one}. Under Assumption \ref{ass:one}, we can define the 
operator $\topop_\ell:E^{\otimes \ell}\to E^{\otimes \ell}$ by
$$
\topop_\ell\nu:=\tilde{\nu}=\lim_{n\to\infty}\mathrm{e}^{-\beta_n}\nu \mathrm{e}^{\beta_{n-\ell}}.
$$
The map $Z(A)\otimes Z(A^\op)\to 
\End_A^*(E_A^{\otimes \ell})\cap \End^*_A({}_AE^{\otimes \ell})$ defined by $(a_{1}\otimes a_{2}^\op)e:=a_{1}ea_{2}$ is an
injective $*$-homomorphism into the algebra of 
left and right adjointable operators on $E^{\otimes \ell}$. In view of this, the definition of $\topop_\ell$ immediately yields the following.
\begin{lemma}
\label{lem:adj-lim}
The operator $\topop_\ell:E^{\otimes \ell}\to E^{\otimes \ell}$ does not depend on the choice of frame. 
Moreover $\topop_\ell$ is adjointable and positive with respect to both left and right
inner products, and in particular $\topop_\ell\in Z(A)\otimes Z(A^\op)\subset \End_A^*(E_A^{\otimes \ell})\cap \End^*_A({}_AE^{\otimes \ell})$. 
\end{lemma}

\begin{proof}
The operator $\Phi_\ell$ and the element $\mathrm{e}^{\beta_\ell}=\Phi_\ell(1)$ are independent of choice of frame, so 
therefore $\topop_\ell$ is independent of choice of frame.
The fact that $\topop_\ell$ is adjointable follows from the adjointability of the left and right actions of $A$, 
and the centrality of $\mathrm{e}^{\beta_n}$. For instance, when $\mu,\nu\in E^{\otimes \ell}$,
\begin{align*} {}_A(\nu|\topop_\ell{\mu})
&=\lim_{n\to\infty}{}_A(\nu|\mathrm{e}^{-\beta_n}\mu \mathrm{e}^{\beta_{n-\ell}})
=\lim_{n\to\infty}{}_A(\nu \mathrm{e}^{\beta_{n-\ell}}|\mu )\mathrm{e}^{-\beta_n}\\
&=\lim_{n\to\infty}\mathrm{e}^{-\beta_n}{}_A(\nu \mathrm{e}^{\beta_{n-\ell}}|\mu )
=\lim_{n\to\infty}{}_A(\mathrm{e}^{-\beta_n}\nu \mathrm{e}^{\beta_{n-\ell}}|\mu )
={}_A(\topop_\ell\nu |\mu ),
\end{align*}
and the proof for the right inner product is similar. That $\topop_\ell\in Z(A)\otimes Z(A^{\op})$ 
and is positive follows because it is the
limit of positive operators in $Z(A)\otimes Z(A^{\op})$.
\end{proof}

In addition to the fact that $\topop_\ell$ is a bimodule morphism, the $\topop_\ell$'s are 
multiplicative in the following sense. For $\mu\in E^{\otimes m}$, $\nu\in E^{\otimes \ell}$,
\begin{align}
\topop_{m+\ell}(\nu\ox\mu)&=\lim_{n\to \infty}\mathrm{e}^{-\beta_n}\nu\ox\mu \mathrm{e}^{\beta_{n-m-\ell}}
=\lim_{n\to \infty}\mathrm{e}^{-\beta_n}\nu \mathrm{e}^{\beta_{n-\ell}}\ox \mathrm{e}^{-\beta_{n-\ell}}\mu \mathrm{e}^{\beta_{n-\ell-m}}
=\topop_\ell(\nu)\ox \topop_m(\mu). 
\label{eq:Tee-nice}
\end{align}
The conditional expectation $\Phi_\ell$ applied to $\topop_\ell$ in the tensor power $E^{\ox \ell}$ can 
be computed to be $1_A$:
\begin{align}
\Phi_\ell(\topop_\ell|_{E^{\ox \ell}})&=\sum_{|\rho|=\ell}{}_A(\topop_\ell e_\rho|e_\rho)
=\lim_n\sum_{|\rho|=\ell}{}_A(\mathrm{e}^{-\beta_n}e_\rho \mathrm{e}^{\beta_{n-\ell}}|e_\rho)\nonumber\\
&=\lim_n\sum_{|\rho|=\ell,|\sigma|=n-\ell} \mathrm{e}^{-\beta_n}{}_A(e_\rho\ox e_\sigma|e_\rho\ox e_\sigma)
=1_A.
\label{eq:trace-one}
\end{align}

Regardless of all these properties, we need to impose a further technical 
requirement on the operators $\mathfrak{q}_{\ell}$. We first prove a structural 
result about $\topop_\ell$ assuming that $\topop_1$ has closed range. Given an $A$-bimodule $E$ 
and $c\in Z(A)$, we say that $c$ is \emph{central for the bimodule structure} if for all $e\in E$ the 
equality $ce=ec$ holds
\footnote{Note that a central element in $A$ need not be central for the bimodule structure, 
e.g. in Example \ref{localexamnple}, $a\in C(V)$ is central for the bimodule structure iff $a\circ g=a$.}.

\begin{lemma}
\label{decomposingtopop}
Assume that the range of $\topop_1$ is closed. Then $\topop_{\ell}$ has closed range for any $\ell$. 
Consequently there is an $A$-bilinear projection 
$P_\ell$ on $E^{\otimes \ell}$ such that $\mathfrak{q}_{\ell}$ is invertible on the range of $P_{\ell}$ and 
$
\topop_\ell=\mathfrak{q}_\ell P_\ell.
$
Furthermore, 
\begin{enumerate}
\item[a.] If there exists $c_\ell \in A$ such that $c_{\ell}P_{\ell}=\topop_\ell P_{\ell}$, 
then $c_\ell=\Phi_\ell(P_\ell)^{-1}$ and 
thus $c_\ell$ is invertible and central in $A$. 
\item[b.] If $c_1$ is given by left multiplication by an element in $A$ which is central for the bimodule structure, 
then $c_\ell$ is given by left multiplication by the central invertible element $c_1^\ell\in A$ for all $\ell$.
\end{enumerate}
\end{lemma}

\begin{proof}
The assumption that $\topop_1$ has closed range guarantees that the range is complemented (see \cite[Theorem 3.2]{lancesbook}), 
and $E=\ker(\topop_1)\oplus \textnormal{im}(\topop_{1})$ with 
$M:=\topop_1E$ a sub-bimodule. We let $P_1$ denote the 
projection onto $M$, so $P_{1}$ commutes with $A$.  An 
easy induction using Equation \eqref{eq:Tee-nice} 
shows that $\topop_\ell E^{\otimes \ell}=M^{\otimes \ell}$.
Since $P_1$ is also a bimodule map, the projection onto $M^{\ox \ell}$ is easily seen to be
$P_\ell=P_1\ox P_1\ox\cdots\ox P_1$, and $P_{\ell}$ commutes with $A$ as well.

By definition $P_\ell \topop_\ell=\topop_{\ell}$ and thus $\topop_\ell=\topop_\ell P_\ell$ by taking adjoints. 
Moreover, from the decomposition $\ker(\topop_{\ell})\oplus \textnormal{im }(\topop_{\ell})$, 
we see that $\topop_\ell$ is injective on $\textnormal{im } P_{\ell}$. It is 
surjective for if $x=P_{\ell}y$ then there exists $z$ such that $x=P_{\ell}y=\topop_{\ell}z=\topop_{\ell}P_{\ell}z$. 

If there exists $c_\ell\in A$ whose compression with $P_{\ell}$ 
coincides with $\topop_{\ell}$, that is $\topop_{\ell}=P_{\ell}c_{\ell}P_{\ell}=c_{\ell}P_{\ell}$, 
then $1_A=c_\ell\Phi_\ell(P_\ell)$ so 
$\Phi_\ell(P_\ell)$ is invertible in $A$ and $c_{\ell}=\Phi_{\ell}(P_{\ell})^{-1}$. 
Moreover, $\topop_\ell$ commutes with both actions of $A$ so $c_\ell$
is central in $A$. This proves a. 

If $c_1$ is given by left multiplication by a, necessarily central, element in $A$, 
$c_1=\Phi_1(P_1)^{-1}$. The assumption that $c_1$ is central for the bimodule structure, along with
the fact that $\topop$ and $P_1$ are bimodule maps gives
$$
\topop_\ell=c_\ell P_\ell=\topop_1\ox\cdots\ox\topop_1
=c_1P_1\ox  \cdots\ox c_1P_1=c_1^\ell P_1\ox\cdots\ox P_1=c_1^\ell P_\ell.
$$  
This proves b.
\end{proof}

\begin{rmk}
If we can write $\topop_\ell=c_\ell P_\ell=P_\ell c_\ell$ for an $A$-bilinear projection 
$P_\ell$ on $E^{\otimes \ell}$ and a central invertible element $c_\ell\in A$ then 
it trivially holds that $\topop_\ell$ has closed range.
\end{rmk}

\begin{ass}
\label{ass:two}
For any $\ell$, we can write $\topop_\ell=c_\ell P_\ell=P_{\ell}c_{\ell}$ where 
$P_{\ell}\in \End^{*}_{A}(E^{\otimes \ell})$ is a projection and $c_\ell$ 
is given by left-multiplication by an element in $A$. 
\end{ass}

\begin{rmk} 
\label{centralityandasstwo}
It follows from Lemma \ref{decomposingtopop} that if a decomposition $\topop_{\ell}=c_{\ell}P_{\ell}$ 
of the kind in Assumption \ref{ass:two} exists, it is unique and of a very specific form. 
Indeed, each $c_\ell$ is central in $A$, invertible and $c_\ell=\Phi_\ell(P_\ell)^{-1}$. 
Lemma \ref{decomposingtopop} allows one to check Assumption \ref{ass:two} in practice. 
A sufficient condition for Assumption \ref{ass:two} to hold is that $\topop_{1}$ is closed 
with decomposition $\topop_1=c_{1}P_{1}$ 
for an element $c_{1}\in A$ which is central for the bimodule structure on $E$. 
For instance, if $\beta_1$ is central for the bimodule structure on $E$, 
$c_\ell=\mathrm{e}^{-\beta_\ell}=\mathrm{e}^{-\ell\beta_1}$ is central for the bimodule structure on $E$ 
and $P_\ell={\rm Id}_{E^{\ox \ell}}$. We remark that 
it is unclear if the property $c_1\in A$ suffices to guarantee that Assumption \ref{ass:two} holds.
\end{rmk}

\begin{ex}
Graph $C^*$-algebras defined from a primitive graph satisfy Assumption \ref{ass:two} 
by \cite[Equation 3.7]{RRS}. Cuntz-Krieger algebras trivially satisfy Assumption \ref{ass:two} 
because $\beta_\ell=0$ for all $\ell$ in this case and $\topop_\ell=\mathrm{Id}_{E^{\otimes \ell}}$ (see Example \ref{betakcomp}).
\end{ex}

{\bf We assume that Assumption \ref{ass:two} holds for the remainder of the paper.}

To simplify notations, we write $P=\sum_\ell P_\ell$ interpreted as a strict sum. 
We also write $\mathfrak{q}=\oplus_\ell\mathfrak{q}_\ell$ which we interpret as 
a densely defined operator with domain $\algFock$. 
We can now turn to generating the direct sum decomposition of $\phimod$.
We recall the notation $W_{\mu,\nu}$ for the class of $S_\mu S_\nu^*$ in $\phimod$.
\begin{lemma}
\label{wmunuwm}
For all homogeneous $\mu,\,\nu\in \algFock$, $W_{\mu,\nu}=W_{\mu,P\nu}$ in $\phimod$.
\end{lemma}
\begin{proof}
We compute the module norm of the difference of $[S_\mu S_\nu^*]=W_{\mu,\nu}$
and $[S_\mu S_{P{\nu}}^*]=W_{\mu,P{\nu}}$ and show that it is zero. 
Write $\ell:=|\nu|$. Using Equation \eqref{eq:mu-nu-beta} and the definition of $\topop_\ell$ we have
\begin{align*}
&(W_{\mu,\nu}-W_{\mu,P{\nu}}|W_{\mu,\nu}-W_{\mu,P{\nu}})_A\\
&=\Phi_\infty\left(S_\nu(\mu|\mu)_AS_\nu^*-S_{P{\nu}}(\mu|\mu)_AS_\nu^*
-S_\nu(\mu|\mu)_AS_{P{\nu}}^*+S_{P{\nu}}(\mu|\mu)_AS_{P{\nu}}^*\right)\\
&={}_A(\nu(\mu|\mu)_A|\topop_\ell{\nu})-{}_A(P{\nu}(\mu|\mu)_A|\topop_\ell{\nu})
-{}_A(\nu(\mu|\mu)_A|\topop_\ell P{\nu})+{}_A(P{\nu}(\mu|\mu)_A|\topop_\ell P{\nu})=0.\qedhere
\end{align*}
\end{proof}

The next result shows that 
if $\{e_1,e_2,\ldots, e_N\}$ is a frame for $E$ as a right module, then $\{P_ke_\rho\}_{|\rho|=\ell}$
is a frame for $M^{\ox \ell}=P_\ell E^{\ox \ell}$, and similarly for left frames $\{f_1,f_2,\ldots, f_M\}$.
We just state the result for the left frame, as this is all we will require below.

\begin{lemma}
Let $f_1,\dots,f_M$ be a frame for the left module ${}_AE$. Then
with $Pf_\rho=P_1f_{\rho_1}\ox\cdots\ox P_1f_{\rho_\ell}$ we have for
all $\mu \in E^{\otimes \ell}$
\begin{equation}\label{eq:framebehaviour}
\sum_{|\rho|=\ell}{}_A(\mu |Pf_{\rho})f_\rho=P\mu=
\sum_{|\rho|=\ell}{}_A(P\mu|f_\rho)f_\rho=
\sum_{|\rho|=\ell}{}_A(\mu|f_{\rho})Pf_\rho=
\sum_{|\rho|=\ell}{}_A(\mu|Pf_{\rho})Pf_\rho.
\end{equation}
\end{lemma}

\begin{proof}
We use Lemma \ref{lem:adj-lim} to compute
\begin{align*}
\sum_{|\rho|=\ell}{}_A(\mu|Pf_{\rho})f_\rho
&=\sum_{|\rho|=\ell}{}_A(P\mu|f_{\rho})f_\rho
=P\mu
\end{align*}
since $\{f_\rho\}$ is a frame for $E^{\ox \ell}$. Finally
$$
P\mu=P\Big(\sum_{|\rho|=\ell}{}_A(\mu|f_\rho)f_\rho\Big)=\sum_{|\rho|=\ell}{}_A(\mu|f_\rho)Pf_\rho,
$$
since $P$ is a bimodule map. Applying this identity to $P(P\mu)=P\mu$ gives the last expression.
\end{proof}

We want to build a frame for  the module $\phimod$. First, we identify the rank one 
operators we need. Recall the notational convention for simple tensors on page \pageref{conventionforsimple}. 
We start with the main computational step.

\begin{lemma} 
\label{lemmatwopointnine}
For $\mu,\nu\in \algFock$, write $n:=|\mu|-|\nu|$. 
For multi-indices $\rho,\sigma$ with $|\rho|=k+n, |\sigma|=k$, we have the following identities:
\begin{equation}
\label{eq:fdup}
W_{e_\rho,c^{-1/2}_{|\sigma|}P{f}_\sigma}
\Phi_\infty(S_{c^{-1/2}_{|\sigma|}P{f}_\sigma}S^*_{e_\rho}S_{\mu}S^*_{\nu}) =\left\{\begin{array}{ll}
W_{e_\rho,{}_A(\underline{P\nu}\,{}_A(\mathfrak{q}\overline{{\nu}}|\overline{\mu})
(\underline{\mu}|e_\rho)_A|P{f}_\sigma)P{f}_\sigma},
& k+n\leq |\mu|,\\ 
\vspace{-2mm}\\
W_{e_\rho,{}_A(P({\nu}(\mu|e_{\underline{\rho}})_Ae_{\overline{\rho}})|P{f}_\sigma)P{f}_\sigma},
& k+n\geq |\mu|.
\end{array}\right.
\end{equation}
\end{lemma}

\begin{proof} 
We first treat the case $k+n\leq |\mu|$ and compute 
\begin{align*}
W_{e_\rho,c^{-1/2}_{|\sigma|}P{f}_\sigma}
\Phi_\infty(S_{c^{-1/2}_{|\sigma|}P{f}_\sigma}S^*_{e_\rho}S_{\mu}S^*_{\nu}) & =
W_{e_\rho,c^{-1/2}_{|\sigma|}P{f}_\sigma}
\Phi_\infty(S_{c^{-1/2}_{|\sigma|}P{f}_\sigma}(e_\rho|\underline{\mu})_AS_{\overline{\mu}} S_\nu^*)\\
&=W_{e_\rho,P{f}_\sigma}\Phi_{\infty}(S_{c^{-1}_{|\sigma|}Pf_{\sigma}(e_{\rho}|\underline{\mu})_{A}\overline{\mu}}S_{\nu}^{*})\\
&=\lim_{\ell\to\infty} W_{e_\rho,P{f}_\sigma}{}_{A}
(c^{-1}_{|\sigma|}Pf_{\sigma}(e_{\rho}|\underline{\mu})_{A}\overline{\mu} | \mathrm{e}^{-\beta_{\ell}}\nu \mathrm{e}^{\beta_{\ell-|\nu|}})
\quad\textnormal{by Eq. \eqref{eq:mu-nu-beta}}\\
&=W_{e_\rho,P{f}_\sigma}{}_A(P{f}_\sigma(e_\rho|\underline{\mu})_A\overline{\mu}|c^{-1}_{|\sigma|}\mathfrak{q}\nu) \\
&=W_{e_\rho,{}_A(\underline{P\nu}\,{}_A(\mathfrak{q}\overline{{\nu}}|\overline{\mu})
(\underline{\mu}|e_\rho)_A|P{f}_\sigma)P{f}_\sigma}.
\end{align*}
For $k+n\geq |\mu|$ we can do a similar calculation
\begin{align*}
W_{e_\rho,c^{-1/2}_{|\sigma|}P{f}_\sigma}
\Phi_\infty(S_{c^{-1/2}_{|\sigma|}P{f}_\sigma}S^*_{e_\rho}S_{\mu}S^*_{\nu}) 
&= W_{e_\rho,c^{-1/2}_{|\sigma|}P{f}_\sigma}\Phi_\infty(S_{c^{-1/2}_{|\sigma|}P{f}_\sigma}S_{e_{\overline{\rho}}}^*(e_{\underline{\rho}}|\mu)_AS_\nu^*)\\
&=W_{e_\rho,P{f}_\sigma}\Phi_{\infty}(S_{c^{-1}_{|\sigma|}P{f}_\sigma}S^{*}_{\nu(\mu|e_{\underline{\rho}})_{A}e_{\overline{\rho}}})\\
&=\lim_{\ell\to\infty}{}_{A}(c^{-1}_{|\sigma|}P{f}_\sigma| \mathrm{e}^{-\beta_{\ell}}\nu(\mu|e_{\underline{\rho}})_{A}e_{\overline{\rho}}\mathrm{e}^{-\beta_{\ell-|\nu|-|\overline{\rho}|}}) \quad\textnormal{by Eq. \eqref{eq:mu-nu-beta}}\\
&=W_{e_\rho,P{f}_\sigma}{}_A(P{f}_\sigma(e_\rho|\underline{\mu}){}_A(P{f}_\sigma|c^{-1}_{|\sigma|}\mathfrak{q}(\nu (\mu|e_{\underline{\rho}})_Ae_{\overline{\rho}}))\\
&=W_{e_\rho,{}_A(P({\nu}(\mu|e_{\underline{\rho}})_Ae_{\overline{\rho}})|P{f}_\sigma)P{f}_\sigma},
\end{align*}
establishing the desired formula.
\end{proof}

Lemma \ref{lemmatwopointnine} puts us in a position to define the elements of the decomposing frame for the module $\Xi_{A}$.

\begin{lemma}
\label{lem:baby-step1}
For $n\in\Z$ and $k\geq \max\{0,-n\}$ we define 
$$Q_{n,k}:=\sum_{|\rho|-|\sigma|=n,\,|\sigma|=k}
\Theta_{W_{e_\rho,c^{-1/2}_{|\sigma|}Pf_\sigma},W_{e_\rho,c^{-1/2}_{|\sigma|}Pf_\sigma}}.$$
We have the identity
$$
Q_{n,k}W_{\mu,\nu}
=\delta_{|\mu|-|\nu|,n}\left\{\begin{array}{ll} 
W_{\underline{\mu}\,{}_A(\overline{\mu}|\mathfrak{q}\overline{{\nu}}),P\underline{{\nu}}} & k\leq |\nu|\\
W_{\mu,P{\nu}} & k\geq |\nu|\end{array}\right.
$$
where $|\underline{\nu}|=k$ and $|\underline{\mu}|=n+k$. In particular, $Q_{n,k}$ does not depend on the choice of frames.
\end{lemma}

\begin{proof}
In the interests of avoiding at least one subscript, 
we write $P$ generically for the projection onto $M^{\ox \ell}=P_\ell E^{\ox \ell}$,
and similarly $\mathfrak{q}$ for $\mathfrak{q}_\ell$.
For simplicity we write $r=k+n$ and note that $|\nu|\geq k$ is equivalent to $|\mu|\geq r$ 
for $|\mu|-|\nu|=n$.
\begin{align*}
&\sum_{|\rho|-|\sigma|=n,\,|\rho|=r}
\Theta_{W_{e_\rho,c^{-1/2}_{|\sigma|}P{f}_\sigma},W_{e_\rho,c^{-1/2}_{|\sigma|}P{f}_\sigma}}W_{\mu,\nu}\\
&=\delta_{|\mu|-|\nu|,n}\sum_{|\rho|-|\sigma|=n,\,|\rho|=r}W_{e_\rho,c^{-1/2}_{|\sigma|}P{f}_\sigma}
\Phi_\infty(S_{c^{-1/2}_{|\sigma|}P{f}_\sigma}S^*_{e_\rho}S_{\mu}S^*_{\nu})\\
&=\delta_{|\mu|-|\nu|,n}\sum_{|\rho|-|\sigma|=n,\,|\rho|=r}
\left\{\begin{array}{ll}
W_{e_\rho,{}_A(\underline{P\nu}\,{}_A(\mathfrak{q}\overline{{\nu}}|\overline{\mu})
(\underline{\mu}|e_\rho)_A|P{f}_\sigma)P{f}_\sigma}
& r\leq |\mu| \quad\textnormal{by Eq.\eqref{eq:fdup}}\\ 
W_{e_\rho,{}_A(P({\nu}(\mu|e_{\underline{\rho}})_Ae_{\overline{\rho}})|P{f}_\sigma)P{f}_\sigma}
& r\geq |\mu|\end{array}\right.\\
&=\delta_{|\mu|-|\nu|,n}\sum_{|\rho|=r}
\left\{\begin{array}{ll}
W_{e_\rho(e_\rho|\underline{\mu})_A{}_A(\overline{\mu}|\mathfrak{q}\overline{{\nu}}),\underline{P{\nu}}}
& r\leq |\mu| \quad\textnormal{by Eq. \eqref{eq:framebehaviour}}\\ 
W_{e_\rho, P{\nu}(\mu|e_{\underline{\rho}})_AP{e}_{\overline{\rho}}}
& r\geq |\mu|\end{array}\right.\\
&=\delta_{|\mu|-|\nu|,n}
\left\{\begin{array}{ll} W_{\underline{\mu}{}_A(\overline{\mu}|\mathfrak{q}\overline{\nu}),\underline{P{\nu}}}
& r\leq |\mu|\qquad \qquad\mbox{where}\ \ |\underline{\mu}|=r\\ 
\sum_{|\rho|=r} W_{e_\rho, P{\nu}(\mu|e_{\underline{\rho}})_AP{e}_{\overline{\rho}}}
& r\geq |\mu| \end{array}\right.\\
&=\delta_{|\mu|-|\nu|,n}\left\{\begin{array}{ll}
W_{\underline{\mu}{}_A(\overline{\mu}|\mathfrak{q}\overline{\nu}),\underline{P{\nu}}}
& r\leq |\mu|\qquad \qquad\mbox{where}\ \ |\underline{\mu}|=r\\ 
W_{\mu,P{\nu}}
& r\geq |\mu|\end{array}\right.
\end{align*}
In the last step, we are using $W_{\mu,\nu}=W_{\mu,P\nu}$ 
(see Lemma \ref{wmunuwm}) to find that
\begin{align*}
\sum_{|\rho|=r}W_{e_\rho, P{\nu}(\mu|e_{\underline{\rho}})_AP{e}_{\overline{\rho}}}
&=\sum_{|{\rho}|=r}W_{e_{\rho}, {\nu}(\mu|e_{\underline{\rho}})_A{e}_{\overline{\rho}}}
=\sum_{|\rho|=r}S_{e_\rho}S^*_{e_{\overline{\rho}}}W_{0,{\nu}(\mu|e_{\underline{\rho}})_A}\\
&=\sum_{|\underline{\rho}|=|\mu|}W_{e_{\underline{\rho}}, {\nu}(\mu|e_{\underline{\rho}})_A}
=\sum_{|\underline{\rho}|=|\mu|}W_{e_{\underline{\rho}}(e_{\underline{\rho}}|\mu)_A, {\nu}}
=W_{\mu, {\nu}}
=W_{\mu, P{\nu}}.
\end{align*}
Our computation of $Q_{n,k}W_{\mu,\nu}$ gives a result that 
does not depend on the choice of frame,
therefore $Q_{n,k}$ does not depend on the choice of frame.
\end{proof}

With the collection of rank one operators from Lemma \ref{lem:baby-step1} in hand, we can construct our 
orthogonal decomposition. Recall the definition of $\Psi_n$ from Lemma \ref{wherewefindpsin}.

\begin{prop}
\label{lem:baby-step2}
For $n\in\Z$ and $k\geq \max\{0,-n\}$, the operators $Q_{n,k}$ from Lemma \ref{lem:baby-step1} 
have the strict limit
$$
\lim_{k\to\infty}Q_{n,k}=\Psi_n
$$
In particular, we have
$$
Q_{n,\max\{0,-n\}}+\sum_{k=\max\{0,-n\}}^\infty(Q_{n,k+1}-Q_{n,k})=\Psi_n,
$$
again strictly. Finally, the operators
$$
P_{n,k}:=
\left\{\begin{array}{ll} Q_{n,k}-Q_{n,k-1} & k>\max\{0,-n\}\\ Q_{n,\max\{0,-n\}} & k=\max\{0,-n\}
\end{array}
\right.
$$
define a family of pairwise orthogonal finite rank projections which sum (strictly) to the identity
on $\phimod$. The family of projections $P_{n,k}$ does not depend on the choice of frame.
\end{prop}

\begin{proof} It follows from Lemma \ref{lem:baby-step1} that each $Q_{n,k}$ is 
bounded, self-adjoint and idempotent, 
and thus $\|Q_{n,k}\|\leq 1$. As above, we write $r=n+k$. Since for $r\geq |\mu|$ we have 
$$
\sum_{|\rho|-|\sigma|=n,\,|\rho|=r}
\Theta_{W_{e_\rho,c^{-1/2}_{|\sigma|}P{f}_\sigma},W_{e_\rho,c^{-1/2}_{|\sigma|}P{f}_\sigma}}W_{\mu,\nu}
=\Psi_nW_{\mu,\nu},
$$
the strict convergence statements follow now because the sequence $Q_{n,k}$ is uniformly bounded 
and the set $\{W_{\mu,\nu}: \mu,\nu\in \algFock\}$ span a dense submodule.

The second statement follows from the first by a telescoping argument. We are left with the third
statement, and we begin by showing that the $P_{n,k}$ are in fact projections. Since $P_{n,k}$ is a difference
of self-adjoints, $P_{n,k}$ is self-adjoint.
As mentioned above, for $P_{n,\max\{0,-n\}}=Q_{n,\max\{0,-n\}}$ it follows directly 
from Lemma \ref{lem:baby-step1} that $P_{n,\max\{0,-n\}}$ is idempotent. We now turn to the generic case $k>\max\{0,-n\}$.

To reduce the number of subscripts, we drop the subscript on $\topop$. 
Using the computations in Lemma \ref{lem:baby-step1} and writing $r=n+k$, we have
\begin{equation}
P_{n,k+1}W_{\mu,\nu}=\delta_{|\mu|-|\nu|,n}\left\{\begin{array}{ll}
W_{\underline{\mu}_{r+1}{}_A(\overline{\mu}^{r+1}|\topop\overline{{\nu}}^{r+1}),\underline{P{\nu}}_{r+1}}
-W_{\underline{\mu}_{r}{}_A(\overline{\mu}^{r}|\topop\overline{{\nu}}^{r}),\underline{P{\nu}}_r}
& |\mu|\geq r+1\\
0 & |\mu|\leq r
\end{array}\right.
\label{eq:pee-en-ar}
\end{equation}
where $\underline{\mu}_r$ is the initial segment of length $r$, and $\overline{\mu}^r$ has length
$|\mu|-r$. Now this computation shows that
$$
P_{n,k+1}W_{\underline{\mu}_{r}{}_A(\overline{\mu}^{r}|\topop\overline{{\nu}}^{r}),\underline{P{\nu}}_r}
=0
$$
and that with $\mu=\underline{\mu}_r\mu_{r+1}\overline{\mu}^{r+1}$
\begin{align}
P_{n,k+1}
W_{\underline{\mu}_{r+1}{}_A(\overline{\mu}^{r+1}|\topop\overline{{\nu}}^{r+1}),\underline{P{\nu}}_{r+1}}
&=W_{\underline{\mu}_{r+1}{}_A(\overline{\mu}^{r+1}|\topop\overline{{\nu}}^{r+1}),\underline{P{\nu}}_{r+1}}
-W_{\underline{\mu}_r{}_A(\mu_{r+1}{}_A(\overline{\mu}^{r+1}|\topop\overline{{\nu}}^{r+1})|\tilde{\nu}_{r+1}),\underline{P{\nu}}_r}\nonumber\\
&=W_{\underline{\mu}_{r+1}{}_A(\overline{\mu}^{r+1}|\topop\overline{{\nu}}^{r+1}),\underline{P{\nu}}_{r+1}}
-W_{\underline{\mu}_r{}_A(\overline{\mu}^{r}|\topop\overline{{\nu}}^{r}),\underline{P{\nu}}_r}
=P_{n,k+1}W_{\mu,\nu}
\label{eq:pee-skwared}
\end{align}
whence $P_{n,k+1}^2=P_{n,k+1}$. The projection property of $P_{n,k}$ implies, 
by a standard algebraic computation, that $Q_{n,k}Q_{n,k-1}=Q_{n,k-1}Q_{n,k}=Q_{n,k-1}$, 
and by induction for $l<k$ $Q_{n,k}Q_{n,l}=Q_{n,l}$. 
The pairwise orthogonality of the $P_{n,k}$ is now immediate.
\end{proof}

\subsection{Examples of the orthogonal decomposition}
\label{subsec:dec-exam}

We will in this subsection compute some examples of the orthogonal decomposition defined  
from the projections in Lemma \ref{lem:baby-step2}. First we consider Cuntz-Krieger algebras.

\begin{lemma}
\label{descirbingdeom}
Let $\pmb{A}$ denote an $N\times N$-matrix of $0$'s and $1$'s, $\mathcal{R}_{\pmb{A}}$ the associated 
groupoid as in Equation \eqref{thegroupoidra} decomposed as in Equation \eqref{thegroupoidradecomp}.
Under the isomorphism $\Xi_{C(\Omega_{\pmb{A}})}\cong L^{2}(\mathcal{R}_{\pmb{A}})_{C(\Omega_{\pmb{A}})}$, 
$$
C(\mathcal{R}^{n,k}_{\pmb{A}})=P_{n,k}\Xi_{C(\Omega_{\pmb{A}})},
$$
as Hilbert $C^*$-modules.
\end{lemma}

\begin{proof} 
For support reasons, 
$L^2(\mathcal{R}_{\pmb{A}})_{C(\Omega_{\pmb{A}})}=\bigoplus_{n,k} C(\mathcal{R}_{\pmb{A}}^{n,k})$ is an 
orthogonal decomposition. Therefore, the theorem follows if we can prove that $P_{n,k}f=f$ for 
$f\in C(\mathcal{R}^{n,k}_{\pmb{A}})$.

Let $1^{\otimes j}\in E^{\otimes j}$ correspond to the constant function $1$ under 
$E^{\otimes j}\cong C(\Omega_{\pmb{A}})$, which is a left frame by Example \ref{betakcomp}.
To shorten notation, we write $W_{\rho,j}=W_{e_\rho,1^{\otimes j}}$ -- these are the elements 
appearing in $Q_{n,k}$ (see Lemma \ref{lem:baby-step1}, cf. Example \ref{betakcomp}). 
We can identify $W_{\rho,j}$ with functions on $\mathcal{R}_{\pmb{A}}$ given by
$$W_{\rho,j}(x,n,y)=\begin{cases}
\chi_\rho(x), \;&\mbox{if}\; n=|\rho|-j\;\mbox{and} \; \sigma^{|\rho|}(x)=\sigma^{j}(y)\\
0, \;&\mbox{otherwise}. \end{cases}
$$

For $f\in C(\mathcal{R}^{n,k}_{\pmb{A}})\subseteq C_c(\mathcal{R}_{\pmb{A}})$, we compute that
$$\langle W_{\rho,j},f\rangle_{L^2(\mathcal{R}_{\pmb{A}})}(y)=
\sum_{(z,n,y)\in \mathcal{R}_{\pmb{A}}} W_{\rho,j}(z,n,y)f(z,n,y)=\sum_{\sigma^j(y)=
\sigma^{|\rho|}(z)} \chi_\rho(z)f(z,n,y).$$
This can be combined into 
\begin{align*}
Q_{n,k} f(x,n,y)&=
\sum_{|\rho|-j=n,\,|\rho|=n+k}W_{\rho,j}(x,n,y)\langle W_{\rho,j},f\rangle_{L^2(\mathcal{R}_{\pmb{A}})}(y)\\
&=\begin{cases}
\sum_{|\rho|-j=n,\,|\rho|=n+k}\sum_{\sigma^j(y)=\sigma^{|\rho|}(z)} \chi_\rho(x)\chi_\rho(z)f(z,n,y),\; &\mbox{for}\; \sigma^j(y)=\sigma^{n+k}(x)\\
0, &\mbox{otherwise} \end{cases}\\
&=\begin{cases}
\sum_{|\rho|-j=n,\,|\rho|=n+k}\sum_{\sigma^k(y)=\sigma^{n+k}(z)} \chi_\rho(x)\chi_\rho(z)f(z,n,y),\; &\mbox{for}\; \sigma^j(y)=\sigma^{n+k}(x)\\
0, &\mbox{otherwise} \end{cases}.
\end{align*}
In the last identity we used the fact that on the support of $f$, $\kappa_{\pmb{A}}(z,n,y)=k$. If $k=0$, 
then $n\geq 0$ and $j=0$. The injectivity of $\sigma$ on $U_j$ implies that 
$$
Q_{n,0} f(x,n,y)=P_{n,0}f(x,n,y)=f(x,n,y).
$$

When considering $k>0$, we write
\begin{align}
\label{writingoutpnr}
P_{n,k+1}f(x,n,y)=&Q_{n,k+1}f(x,n,y)-Q_{n,k}f(x,n,y)\\
\nonumber
=&\sum_{|\rho|-j=n,\,|\rho|=n+k+1}\sum_{\sigma^k(y)=
\sigma^{n+k}(z)} \chi_\rho(x)\chi_\rho(z)f(z,n,y)\delta_{\sigma^{j+1}(y)=\sigma^{n+k+1}(x)}\\
\nonumber
&\quad-\sum_{|\rho|-j=n,\,|\rho|=n+k}\sum_{\sigma^k(y)=
\sigma^{n+k}(z)} \chi_\rho(x)\chi_\rho(z)f(z,n,y)\delta_{\sigma^{j}(y)=\sigma^{n+k}(x)}.
\end{align}
Again using the injectivity of $\sigma$ on $U_j$, $P_{n,k+1}f(x,n,y)=f(x,n,y)$ follows from 
Equation \eqref{writingoutpnr} using a case-by-case analysis.
\end{proof}

We turn to the case of the Cuntz algebra. The case of a general graph $C^*$-algebra 
is combinatorially complicated, especially in light of the discussion in Example \ref{discongrap}.

\begin{lemma}
\label{decomposinl2on}
The GNS-representation $L^2(O_N)$ of the Cuntz algebra $O_N$ associated 
with the KMS-state coincides with $\Xi_\C$ defined from $O_N=O_{\C^N}$. 
Moreover, $L^2(O_N)$ can be decomposed into orthogonal 
finite-dimensional subspaces given by
$$
\mathcal{H}_{n,k}=P_{n,k}\Xi_\C=\begin{cases}
\textnormal{span}\{N W_{\mu i,\nu j}-W_{\mu,\nu}\delta_{i,j}: \; |\nu|=k-1,\, |\mu|-|\nu|=n,\,i,\,j=1,\dots,N\}, &n+k,k>0\\
\textnormal{span}\{ W_{\mu ,\emptyset}: \; |\mu|=n\},&k=0\\
\textnormal{span}\{ W_{\emptyset ,\nu}: \; |\nu|=-n\},&n+k=0
\end{cases}\!\!\!.
$$
\end{lemma}

\begin{proof}
We write $r=n+k$. It follows from construction that 
$$
Q_{n,k}W_{\mu,\nu}=
\delta_{|\mu|-|\nu|,n}
\left\{\begin{array}{ll} 
W_{\underline{\mu},\underline{{\nu}}}\delta_{\overline{\mu},\overline{\nu}} N^{r-|\mu|} 
& r\leq |\mu|\\
W_{\mu,\nu} & r\geq |\mu|\end{array}\right..
$$
Here we use the convention $|\underline{\mu}|=r=n+k$ as in 
Lemma \ref{lem:baby-step1}. From this, it follows that the 
space where $(c,\kappa)=(n,0)$ and $(c,\kappa)=(n,-n)$ is exactly $\mathcal{H}_{n,0}$ 
and $\mathcal{H}_{n,-n}$, respectively. 
We also have the identity
\begin{align*}
P_{n,k+1}W_{\mu,\nu}&=Q_{n,k+1}W_{\mu,\nu}-Q_{n,k}W_{\mu,\nu}\\
&=\delta_{|\mu|-|\nu|,n}\left\{\begin{array}{ll} 
N^{r-|\mu|} (NW_{\underline{\mu}_{r+1},\underline{{\nu}}_{r+1}}\delta_{\overline{\mu}^{r+1},\overline{\nu}^{r+1}}-W_{\underline{\mu}_{r},\underline{{\nu}}_{r}}\delta_{\overline{\mu}^{r},\overline{\nu}^{r}}) & r\leq |\mu|\\
W_{\mu,\nu} & r\geq |\mu|\end{array}\right.
\end{align*}
From this computation, we conclude that the space 
where $(c,\kappa)=(n,k)$ is exactly $\mathcal{H}_{n,k}$ for $r, k>0$.
\end{proof}

\subsection{The depth-kore operator and the unbounded Kasparov module}
\label{subsec:kas-mod}

We fix the decomposition 
\begin{equation}
\label{phimoddec}
\phimod=\bigoplus_{n\in\Z}\bigoplus_{k\geq\max\{0,-n\}}P_{n,k}\phimod
\end{equation}
established in the last section.
We consider the $A$-linear operators defined on the algebraic direct sum 
$\oplus_{k\geq \max\{0,-n\}}^{\textnormal{alg}}P_{n,k}\phimod\subseteq \phimod$ by the formulae
$$
\kappa_0:=\sum_{n,k} kP_{n,k}\quad\mbox{and}\quad c_0:=\sum_{n} n\Psi_{n}.
$$
Both $\kappa_0$ and $c_0$ are closable. We define $\kappa$ and $c$ as the closures of 
$\kappa_0$ and $c_0$, respectively. We call $\kappa$ the \emph{depth-kore operator}. 
The following Proposition is immediate from the construction.

\begin{prop}
\label{propertiesofck}
The operators $c$ and $\kappa$ are self-adjoint and regular operators on $\phimod$ such that 
$c+\kappa\geq 0$ on $\Dom(c)\cap \Dom(\kappa)$. They commute on the common core 
$\Dom(c\kappa)=\Dom(\kappa c)$.
\end{prop}

\begin{defn}[cf. Equation (5.37) of \cite{GM}]
\label{psidef}
For $k\geq\max\{0,-n\}$, we define the function $\psi:\Z\times \N\to \Z$ by
$$
\quad\psi(n,k)
=\left\{\begin{array}{ll} n & k=0\\ -(k+|n|) & \mbox{otherwise}
\end{array}\right.
$$
We define $\D:=\psi(c,\kappa)=\sum_{n,k}\psi(n,k)P_{n,k}$ as a densely defined operator on $\phimod$.
\end{defn}

\begin{rmk}
\label{psichange}
We remark that $\D$ does not depend on the choice of frame. However, 
there is freedom in the choice of function $\psi$ used to assemble $\D$ from $c$ and $\kappa$. 
We will see below that the bounded commutator calculation boils down to
some relatively simple estimates. This gives us some freedom in choosing 
the function $\psi$. There are reasons for preferring the definitions
$$
\psi(n,k)
=\left\{\begin{array}{ll} n & k=0\\ -\frac{1}{2}(k+n+2|n|) & \mbox{otherwise}
\end{array}\right.
\qquad
\psi(n,k)
=\left\{\begin{array}{ll} n & k=0\\ -\frac{1}{2}(k+|n|) & \mbox{otherwise}
\end{array}\right..
$$
The main reason is that these definitions restrict to the number operator
in the SMEB case, since in that case $k=\max\{0,-n\}$ always (see more in Subsection \ref{smebsubse}). 
\end{rmk}

\begin{lemma}
\label{lem:Fock-it}
The projection onto the isometric copy of the Fock space in $\phimod$ is
given by
$$
Q=\sum_{n=0}^{\infty}P_{n,0}.
$$
\end{lemma}

\begin{proof}
This is just $k=0\Rightarrow n\geq 0$.
\end{proof}

\begin{lemma}
\label{lem:mu-n-r}
For all homogeneous $\mu,\alpha,\beta\in \algFock$, $n\in\Z$ we have
$$
S_\mu P_{n,k}W_{\alpha,\beta}=\left\{\begin{matrix}Q_{n+|\mu|,k}W_{\mu\alpha,\beta} & n+k=0,\\ 
P_{n+|\mu|,k}W_{\mu\alpha,\beta} & n+k>0. \end{matrix}\right.
$$
\end{lemma}
\begin{proof} It is straightforward to check the corresponding relations for the $Q_{n,k}$, 
from which the statement of the Lemma follows immediately.
\end{proof}

\begin{thm}
\label{definunbdd}
The data
$(\O_E,\phimod,\D)$ defines an odd unbounded Kasparov module
which represents the class of the extension 
$$
0\to \K_A(\Fock)\to \T_E\to \O_E\to 0
$$
in $KK^1(\O_E,A)$. 
\end{thm}

\begin{proof}
Since $\D$ is given in diagonal form with finitely generated projective eigenspaces, 
the proof of self-adjointness and regularity
is straightforward.
The range of each $P_{n,k}$ is finitely generated, and the function $\psi$ is unbounded
with value $\pm(k+|n|)$ on $P_{n,k}\phimod$, and so $(1+\D^2)^{-1/2}$
is compact. The non-negative spectral projection of $\D$ is precisely
the projection on to the isometric copy of the Fock module in $\phimod$ by
Lemma \ref{lem:Fock-it}, and so if $(\O_E,\phimod,\D)$ is an unbounded 
Kasparov module, its class represents
the extension. 

The only remaining thing to prove is that we have bounded
commutators.
Let $\mu,\,\alpha,\,\beta\in \algFock$ be homogeneous and consider the generator $S_\mu\in \O_E$. 
By Lemma \ref{lem:mu-n-r} we have the computation
\begin{align}
\nonumber
\D S_\mu W_{\alpha,\beta}-S_\mu\D W_{\alpha,\beta}
&=\sum_{n+k>0,n}\psi(n,k)P_{n,k}W_{\mu\alpha,\beta}-S_\mu\sum_{m+l>0,m}\psi(m,l)P_{m,l}W_{\alpha,\beta}\\
\nonumber
&\quad\quad+ \sum_{n}\psi(n,-n)P_{n,-n}W_{\mu\alpha,\beta}-S_\mu\sum_{m}\psi(m,-m)P_{m,-m}W_{\alpha,\beta}\\
\nonumber
&=\sum_{n+k>0,n}\psi(n,k)P_{n,k}W_{\mu\alpha,\beta}
-\sum_{m+l>0,m}\psi(m,l)P_{m+|\mu|,l}W_{\mu\alpha,\beta}\\
\nonumber
&\quad\quad + \sum_{n}\psi(n,-n)Q_{n,-n}W_{\mu\alpha,\beta}-\sum_{m}\psi(m,-m)Q_{m+|\mu|,-m}W_{\mu \alpha,\beta}\\
\label{firstsumcomm}
&=\sum_{m+l>0,m}\big(\psi(m+|\mu|,l)-\psi(m,l)\big)P_{m+|\mu|,l}W_{\mu\alpha,\beta} \\
\label{secondsumcomm}
&\quad\quad +\sum_{n\leq |\mu |}\Big(\psi(n,-n)(Q_{n,-n}-Q_{n+|\mu|,-n})+\sum_{k=0}^{|\mu|-n}\psi(n,k)P_{n,k}W_{\mu\alpha,\beta}\Big)
\end{align}
A case-by-case check shows the identity
$$
\psi(m+|\mu|,l)-\psi(m,l)
=\left\{\begin{array}{ll} |\mu| & m\geq0,\;l=0\\
-|\mu| & m\geq 0,\; l>0\\
-|\mu|+2|m| & m<0,\,m+|\mu|\geq 0\\
|\mu| & m<0,\,m+|\mu|<0.\end{array}\right. .
$$
Therefore the first sum \eqref{firstsumcomm} of the commutator defines a bounded operator. 
For the second two sums \eqref{secondsumcomm}, observe that since $n\leq k\leq |\mu | $ and $k\geq 0$, 
it suffices to address the case $n<0$, in which case $k>0$ and thus $\psi(n,k)=-k+n$ always. 
For fixed $n<0$ we can compute
\begin{align*}
\Big(\psi(n,-n)(Q_{n,-n}-&Q_{n+|\mu|,-n})+\sum_{k=\max\{0,-n\}}^{|\mu|-n}\psi(n,k)P_{n,k}\Big)W_{\mu\alpha,\beta} \\
&=-\Big(2n(Q_{n+|\mu|,-n}-Q_{n,-n})-\sum_{k=-n}^{|\mu|-n}(n-k)P_{n,k}\Big)W_{\mu\alpha,\beta}\\
&=-\sum_{k=-n}^{|\mu|-n}(n+k)P_{n,k}W_{\mu\alpha,\beta}
+2n\Big(Q_{n,-n}-Q_{n+|\mu|,-n}+\sum_{k=-n}^{|\mu|-n}P_{n,k}\Big)W_{\mu\alpha,\beta}\\
&=-\sum_{k=-n}^{|\mu|-n}(n+k)P_{n,k}W_{\mu\alpha,\beta}.
\end{align*}
Thus, since $0\leq n+k\leq|\mu|$, the second two sums \eqref{secondsumcomm} define a bounded operator.
\end{proof}

\begin{rmk}
\label{lipprop}
We see that the crucial properties of $\psi$ for proving that $[\D,S_\mu]$ is bounded 
are that $\psi$ satisfies: for every $l>0$
there is a constant $C_l>0$ such that 
$$
|\psi(n+l,k)-\psi(n,k)|\leq C_l,\quad\forall\, (n,k)\in \Z\times \N, \; n+k\geq 0; \ \mbox{and}
$$
for every $k>0$ there is a constant $C_{k}$ such that
$$
|\psi(n,-n)-\psi(n,k)|\leq C_{k}, \quad \forall\, n\in \Z \;\;\mbox{with}\;\; n\leq -k.
$$
\end{rmk}

\begin{rmk}
An unbounded representative $[\D]$ of the extension class allows one to use the explicit
lift $[\widehat{\D}]$ to the mapping cone of the inclusion $A\hookrightarrow \O_E$
described in \cite{CPR}. This lift allows a concrete comparison, 
on the level of cycles, of the exact sequences determined by the defining extension of a 
Cuntz-Pimsner algebra and the mapping cone exact sequence of $A\hookrightarrow \O_E$.
This comparison is described in \cite{AFR}.
\end{rmk}

\begin{rmk} 
For fixed inner products, the construction of the unbounded Kasparov module is independent of the choices of left and right frames. 
It does however depend heavily on the choices of left and right inner product on the module $E$. 
In certain cases (see the discussion of SMEB's and vector bundles below) 
there is an obvious choice of left inner product, but of course not the only possible choice. 
In general the left inner product is part of the data that goes into the construction.
\end{rmk}

\subsection{Examples of unbounded Kasparov modules and spectral triples}
\label{sec:eggs}

In this subsection we will compute examples and compare to the existing works in the literature. 

\subsubsection{Self Morita equivalence bimodules}
\label{smebsubse}

For a SMEB, the depth-kore operator $\kappa$ takes a 
simple form (cf. Example \ref{smebexample}).

\begin{prop}
\label{smebunitaries}
Suppose that $E$ is a SMEB. The following mapping defines a unitary isomorphism of $A$-modules
$$
\Psi_n:\phimod^n\to E^{\ox n},\quad W_{\mu,\nu}\mapsto 
\begin{cases} 
\underline{\mu}_A(\overline{\mu}|\nu),\;& |\mu|\geq |\nu|\\
{}_A(\mu|\overline{\nu})\overline{\underline{\nu}}, \;& |\mu|< |\nu|,
\end{cases}
$$
where $|\underline{\mu}|=n$ in the first line, $|\underline{\nu}|=-n$ in the second line and in 
the same line $\overline{\underline{\nu}}\in \overline{E}^{\ox |n|}$ denotes the image of 
$\underline{\nu}$ under the anti-linear mapping $E^{\ox |n|}\to \overline{E}^{\ox |n|}$.
\end{prop}

The proof of the previous proposition follows from the proof of \cite[Theorem 3.1]{RRS}. 
From Proposition \ref{smebunitaries} we deduce the structure of the operators $P_{n,k}$ 
from Proposition \ref{lem:baby-step2}. 
We pick a right frame $(e_i)_{i=1}^N$ as in Section \ref{sec:been-done} 
and take $f_j=e_j$ as a left frame. Using the isomorphism of 
Proposition \ref{lem:baby-step2}, we have that
$$
P_{n,k}=
\begin{cases}
\sum_{|\sigma|=n} \Theta_{W_{e_\sigma,\emptyset},W_{e_\sigma,\emptyset}}, \; &k=0,\; n\geq 0,\\
\sum_{|\sigma|=n} \Theta_{W_{\emptyset,e_\sigma},W_{\emptyset,e_\sigma}}, \; &n+k=0,\\
0, &\mbox{otherwise,}
\end{cases}
\ =\ \begin{cases}
\Psi_n, \; &k=0,\; n\geq 0,\\
\Psi_n, \; &n+k=0,\\
0, &\mbox{otherwise.}
\end{cases}
$$
We sum up the consequences for $\kappa$ in a proposition.

\begin{prop}
\label{smebcase}
If $E$ is a SMEB, then 
$$
\kappa=\overline{\sum_{n<0} |n|\Psi_n}=\frac{1}{2}\left(|c|-c\right).
$$
In particular, for $\psi$ and $\D$ as in Remark \ref{psichange} we get $
\D=c,$ the usual number operator.
\end{prop}

\subsubsection{The depth-kore operator $\kappa$ for Cuntz-Krieger algebras}
\label{localhomsubs}

In \cite[Theorem 5.1.7]{GM}, a family of unbounded bivariant $(O_{\pmb{A}},C(\Omega_{\pmb{A}}))$-cycles 
$(O_{\pmb{A}},L^{2}(\mathcal{R}_{\pmb{A}})_{C(\Omega_{\pmb{A}})}, \mathcal{D}_{\lambda})$, 
parameterised by $\lambda$ in the set of finite $\pmb{A}$-admissible words was constructed. 
We let $\circ$ denote the empty word. It was shown that the mapping 
$K^{0}(C(\Omega_{\pmb{A}}))\to K^{1}(O_{\pmb{A}})$ 
defined by taking the Kasparov product with the cycle 
$(L^{2}(\mathcal{R}_{\pmb{A}})_{C(\Omega_{\pmb{A}})},D_{\circ}))$ is surjective. 
We now give a different perspective on this cycle and identify its class. 
The case of general surjective local homeomorphisms is dealt 
with in the context of Smale spaces in \cite{dgmw} by the first two listed authors with 
Robin Deeley and Michael Whittaker.

\begin{thm} 
\label{comparingtogm}
Let $\sigma:\Omega_{\pmb{A}}\to\Omega_{\pmb{A}}$ be a subshift of finite type and 
$E={}_{\Id}C(\Omega_{\pmb{A}})_{\sigma^*}$ the associated $C^{*}$-bimodule.
Under the isomorphism 
$\Xi_{C(\Omega_{\pmb{A}})}\cong L^2(\mathcal{R}_{\pmb{A}})_{C(\Omega_{\pmb{A}})}$, 
the unbounded cycle $(\O_{E}, \Xi_{C(\Omega_{\pmb{A}})}, \D)$ 
constructed in Theorem \ref{definunbdd} 
coincides with the unbounded cycle 
$(O_{\pmb{A}}, L^{2}(\mathcal{R}_{\pmb{A}})_{C(\Omega_{\pmb{A}})},\D_{\circ})$  constructed in 
\cite[Theorem 5.1.7]{GM} for Cuntz-Krieger algebras. 
\end{thm}

The theorem is immediate from the following proposition describing the depth-kore operator. 
The proposition in turn follows from Lemma \ref{descirbingdeom}.
Recall the function $\kappa_{\pmb{A}}\in C(\mathcal{R}_{\pmb{A}})$ from Equation \eqref{kappaadef}.

\begin{prop}
Under the isomorphism $\Xi_A\cong L^{2}(\mathcal{R}_{\pmb{A}})_{C(\Omega_{\pmb{A}})}$, 
the operator $\kappa$ of Proposition 
\ref{propertiesofck} satisfies $C_c(\mathcal{R}_{\pmb{A}})\subseteq \Dom(\kappa)$ 
and for $f\in C_c(\mathcal{R}_{\pmb{A}})$
$$
[\kappa f](x,n,y)=\kappa_{\pmb{A}}(x,n,y)f(x,n,y).
$$
\end{prop}

As a consequence of Theorem \ref{comparingtogm}, we see that the cycle 
$(O_{\pmb{A}}, L^{2}(\mathcal{R}_{\pmb{A}})_{C(\Omega_{\pmb{A}})},\D_{\circ})$ 
represents the extension
\[0\to \K_{C(\Omega_{\pmb{A}})}(\mathcal{F}_{\pmb{A}})\to \T_{\pmb{A}} \to O_{\pmb{A}}\to 0,\]
obtained from $E={}_{\Id}C(\Omega_{\pmb{A}})_{\sigma^*}$ and the isomorphism 
$O_{\pmb{A}}\cong \O_{E}$. Here $\mathcal{F}_{\pmb{A}}$ denotes the Fock module 
constructed from $E={}_{\Id}C(\Omega_{\pmb{A}})_{\sigma^*}$. In particular, 
the unbounded cycle $(O_{\pmb{A}}, L^{2}(\mathcal{R}_{\pmb{A}})_{C(\Omega_{\pmb{A}})},\D_{\circ})$ 
constructed in \cite[Theorem 5.1.7]{GM} represents the boundary maps in the associated 
Pimsner six-term exact sequence in $K$-homology
\begin{equation}
\label{sixexpimforoa}
\xymatrix{K^0(O_{A}) \ar[rr] & & K^0(C(\Omega_{\pmb{A}})) \ar[rr]^{1-[E]} && K^0(C(\Omega_{\pmb{A}})) \ar[d]^{[\D_{\circ}]} \\ 
K^1(C(\Omega_{\pmb{A}})) \ar[u]_{[\D_{\circ}]} & & K^1(C(\Omega_{\pmb{A}})) \ar[ll]_{1-[E]} && K^1(O_{\pmb{A}}) \ar[ll]}
\end{equation} 
Here $[E]\in KK_0(C(\Omega_{\pmb{A}}),C(\Omega_{\pmb{A}}))$ denotes the class associated with 
the bimodule $E$ represented by the (unbounded) Kasparov module $(C(\Omega_{\pmb{A}}),E,0)$. The six term exact 
sequence \eqref{sixexpimforoa} is an example of a Pimsner sequence in $KK$-theory, for further details see \cite{Pimsner}.

Because $\Omega_{\pmb{A}}$ is a compact totally disconnected space, we can compute 
$K^{1}(C(\Omega_{\pmb{A}}))=0$. Thus the sequence \eqref{sixexpimforoa} reduces to
\begin{equation}
\label{pimsnerforck}
0\to K^{0}(O_{\pmb{A}}) \to K^{0}(C(\Omega_{\pmb{A}})) \xrightarrow{1-[E]}K^{0}(C(\Omega_{\pmb{A}})) \xrightarrow{[\D_{\circ}]} K^{1}(O_{\pmb{A}})\to 0.
\end{equation}
We arrive at a conceptual explanation for the surjectivity of  the map 
$K^{0}(C(\Omega_{\pmb{A}})) \xrightarrow{[\D_{\circ}]} K^{1}(O_{\pmb{A}})$ (cf. \cite[Remark 5.2.6]{GM}).
In general, the simple structure of \eqref{pimsnerforck} can not be obtained from the 
Pimsner-Voiculescu sequence (i.e.\ the Pimsner sequence for $E\otimes_A\mathcal{C}_{E}$). The universal coefficient theorem implies that $K^{0}(C(\Omega_{\pmb{A}}))=
\Hom(C(\Omega_{\pmb{A}},\Z),\Z)$ and $[E]$ acts as $\mathfrak{L}^*$. This gives yet another proof of the fact 
$K^{1}(O_{\pmb{A}})=\Z^N/(1-\pmb{A})\Z^N$.

\subsubsection{The two models for $O_N$}
\label{comparingcycletojesus}

 The odd spectral triples on $O_{\pmb{A}}$ constructed in \cite[Theorem 5.2.3]{GM} 
 are supported on the fibres of the groupoid $\mathcal{R}_{\pmb{A}}$. 
 A consequence of Theorem \ref{comparingtogm} is that these Hilbert spaces 
 are localisations of the module $\Xi_{C(\Omega_{\pmb{A}})}$. 
 For the Cuntz algebra $O_{N}$ viewed as a Cuntz-Pimsner algebra over $\C$, 
the method of Subsection \ref{subsec:kas-mod}
will produce a spectral triple. In view of Equation \eqref{KMSCuntz}, 
this spectral triple will be defined on the GNS space $L^{2}(O_{N})$ 
associated to the KMS state. In \cite{GM}, the construction of such spectral triples 
was left as an open problem. We will compare the two approaches in this subsection.
Recall the decomposition $L^2(O_N)=\bigoplus_{n,k} \mathcal{H}_{n,k}$ 
from Lemma \ref{decomposinl2on}.

\begin{thm}
\label{compwithon}
There is a self-adjoint operator $\D$ on $L^2(O_N)$ defined by 
$\D|_{\mathcal{H}_{n,k}}:=\psi(n,k)$ such that 
$(O_N,L^2(O_N),\D)$ is a $\theta$-summable spectral triple on 
$O_N$ whose class generates $K^1(O_N)$.
\end{thm}

\begin{proof}
The exactness of the Pimsner sequence
\begin{equation*}
0\longrightarrow K^0(O_N)\longrightarrow K^0(\C)\stackrel{1-N}{\longrightarrow}K^0(\C)
\stackrel{[\D]\otimes_{\C}\cdot}{\longrightarrow}K^1(O_N)\longrightarrow 0,
\end{equation*}
and Theorem \ref{definunbdd} implies that the class of $(O_N,L^2(O_N),\D)$ 
generates $K^1(O_N)$. The $\theta$-summability of $(O_N,L^2(O_N),\D)$ 
follows from the fact that the dimensions of $\mathcal{H}_{n,k}$ grow exponentially, 
and the sequence $\psi(n,k)$ grows linearly.
\end{proof}

\begin{rmk}
We remark that by the same argument as in the 
proof of \cite[Theorem 5.2.3]{GM}, the bounded Fredholm 
module $(O_N,L^2(O_N),\D|\D|^{-1})$ is $p$-summable for any $p>0$.
\end{rmk}

Following Subsection \ref{localhomsubs}, let $(O_N,\Xi^{\Omega_N},\D^{\Omega_N})$ denote the unbounded Kasparov module that defines 
the Pimsner extension for $O_N$ over 
$C(\Omega_N)$, i.e. a class in $KK_1(O_N,C(\Omega_N))$. 
For $x\in \Omega_N$, we let $\epsilon_x:C(\Omega_N)\to \C$ 
denote the point evaluation at $x$.
We deduce the following result from \cite[Theorem 5.2.3]{GM}.

\begin{corl}
For points $x\in \Omega_N$, the localisations 
$(O_N,\Xi^{\Omega_N}_x,\D^{\Omega_N}_x)=(O_N,\Xi^{\Omega_N}\otimes_{\epsilon_x}\C,\D^{\Omega_N}\otimes_{\epsilon_x}1)$, 
define the same class in $K^1(O_N)$. Moreover, we have
\begin{equation}
\label{facforcla}
[(O_N,L^2(O_N),\D)]=
[(O_N,\Xi^{\Omega_N},\D^{\Omega_N})]\otimes_{C(\Omega_N)} [\epsilon_x]
\quad\mbox{in $K^1(O_N)$}.
\end{equation}
\end{corl}

It can be shown that it is not possible to perform a factorisation as in Equation \eqref{facforcla} at the level of 
unbounded cycles.

\section{The Cuntz-Pimsner algebra of a vector bundle on a closed manifold}
\label{CPalgofbundle}

Our final example is a construction of spectral triples for Cuntz-Pimsner algebras of 
vector bundles on a closed manifold. Let $M$ denote a 
closed Riemannian manifold equipped with 
an $N$-dimensional Hermitian vector bundle $V\to M$ and 
$\phi:M\to M$ a $C^1$-diffeomorphism. We denote the 
induced map by $\alpha:=\phi^{*}:C(M)\to C(M)$ and consider 
the space of continuous sections $_{\alpha}E:=\Gamma(M,V)$ 
as a Hilbert bimodule $_{\alpha}E$ via $(a\cdot f\cdot b)(x)=\alpha(a)(x)f(x)b(x)$. 
The right $C(M)$-valued inner product 
is induced from the Hermitian structure on $V$ and the left 
$C(M)$ valued inner-product is defined through
\[
{}_{C(M)}(f|g):=\alpha^{-1}((g|f)_{C(M)}).
\] 
Because of the close relationship between the left and 
right inner product, we will express all formulae using 
only the right inner product, which will be denoted 
unlabeled, by $(\cdot | \cdot )$. 
Labeled inner products will be used only when necessary.

To work with the module $_{\alpha}\Xi_{C(M)}$, we fix a right frame $(e_{\lambda})_\lambda$ for $_{\alpha}E$
as follows. Consider a finite  open cover $(U_{i})_{i=1}^M$ over which $V$ is trivialised by
$\tau_i:V|_{U_{i}}\to U_{i}\times \C^N$. Choose 
$C^{1}$-functions $\chi_{i}$ such that $(\chi_{i}^{2})_{i=1}^M$ 
is a partition of unity subordinate to $(U_{i})_{i=1}^M$.  
We then take $(\chi_{i}e_j^i)_{i=1\dots M,j=1,\dots, N}$
as our frame, where the collection $(e_j^i)_{j=1,\dots,N}$ 
is an orthonormal basis of $C^1$-sections over
each $U_i$.  

The frame  $(e_{\lambda})_\lambda$ is simultaneously a left and right frame for $_{\alpha}E$ since
\[
\sum_{\lambda} {}_{C(M)}(e|e_{\lambda})e_{\lambda}
=\sum_{\lambda}\alpha^{-1}((e_{\lambda}|e)_{C(M)})\cdot e_{\lambda}
=\sum_{\lambda}e_{\lambda}(e_{\lambda}|e)_{C(M)}=e,\quad e\in {}_{\alpha}E.
\] 
As in the case of the Cuntz algebra $O_{N}$, we have $e^{\beta_{k}}=N^{k}$, 
which is central for the bimodule structure of $E$. Thus it follows that Assumptions 
\ref{ass:one} and \ref{ass:two} are satisfied (see Remark \ref{centralityandasstwo}). 
Moreover, the projections $P_{k}$ are all equal to $1$ and $c_{k}=N^{-k}$. 

\subsection{The product operator}
We will now construct a spectral triple on $\mathcal{O}_{_{\alpha}E}$. 
Let $\Dsla=\Dsla_M$ denote an odd Dirac type operator acting on a 
graded Clifford bundle $S\to M$. We note that it is no restriction to assume that $S$ is graded, 
as our construction 
commutes with tensoring by $\C\ell_1$ (the Clifford algebra of $\C)$, see \cite[Proposition IV.A.13]{connesbook}. The module $_{\alpha}\Xi_{C(M)}$ decomposes as a direct sum $_{\alpha}\Xi_{C(M)}=\oplus_{n,k}{ }_{\alpha}\Xi^{n,k}_{C(M)}$ of finitely generated projective $C(M)$ modules $_{\alpha}\Xi^{n,k}_{C(M)}$ and we denote the associated vector bundles by $_{\alpha}\Xi^{n,k}_{V}\to M$ and the full field of Hilbert spaces by $_{\alpha}\Xi_{V}\to M$. We consider the graded Hilbert space
$$
\mathcal{H}:=\aXi_{C(M)}\otimes_{C(M)}L^2(M,S)=\bigoplus_{k\geq \max\{0,- n\}} L^2(M,\aXi_V^{n,k}\otimes S).
$$
The $C^*$-algebra $\mathcal{O}_{\aE}$ acts on $\mathcal{H}$ via its adjointable 
action on $_{\alpha}\Xi_{C(M)}$. The densely defined operator 
$\D$ on $\Xi_{C(M)}$ and the grading operator $\gamma$ on $S$ induce a densely defined 
self-adjoint operator $D$ on $\mathcal{H}$. The domain of $D$ is 
clearly
$$
\Dom(D):=\left\{f=(f_{n,k})_{n,k}\in  \bigoplus_{k\geq \max\{0,-n\}}L^2(M, \aXi_V^{n,k}\otimes S): 
\sum (|k|+|n|)^2\|f_{n,k}\|_{ L^2(M, \aXi_V^{n,k}\otimes S)}^2<\infty\right\},
$$
and $D(f_{n,r})_{k,n}:=(\D\otimes\gamma)(f_{n,k})_{n,k}=(\gamma\psi(n,k)f_{n,k})_{k,n}$.

To construct a connection on the module $\aXi_{C(M)}$, we observe that by Lemma \ref{lem:baby-step1}, a frame for $Q_{n,k}{}_{\alpha}\Xi_{C(M)}$ is given by $\{N^{k/2}W_{e_{\rho},e_{\sigma}}\}_{|\rho|=n+k,|\sigma |=k}$. For notational convenience we will 
write $W_{\rho,\sigma}:=W_{e_{\rho},e_{\sigma}}$ for multi-indices $\rho,\sigma$. 
The single indices $\iota$ and $\lambda$ will be used in the same way.

\begin{lemma}
\label{globalframe} 
The collection of vectors
\[
x_{\rho,\sigma}:=\left\{\begin{matrix} N^{|\sigma|/2}W_{\emptyset, \sigma} & |\rho|=0\\
W_{\rho, \emptyset} & |\sigma|=0\\
N^{|\sigma|/2}W_{\rho,\sigma}-N^{|\sigma|/2-1}\alpha^{-|\rho|}(e_{\sigma_{|\sigma|}}|e_{\rho_{|\rho|}})W_{\underline{\rho},\underline{\sigma}} &|\rho|>0\ \mbox{and}\ |\sigma|>0\end{matrix}\right.
\]
is a frame for $_{\alpha}\Xi_{C(M)}$. Indeed for fixed $k$ and $n$, $(x_{\rho,\sigma})_{|\sigma|=k,|\rho|-|\sigma|=n}$
forms a frame for $_{\alpha}\Xi^{n,k}_{C(M)}$. 
\end{lemma}

\begin{proof} 
The projections $P_{n,k}\leq Q_{n,k}$ are mutually orthogonal and thus the frame 
$y_{\rho,\sigma}=N^{k/2}W_{e_{\rho},e_{\sigma}}$ for $Q_{n,k}\aXi_{C(M)}$ yields a frame 
for $P_{n,k}\aXi_{C(M)}$ by computing $x_{\rho,\sigma}:=P_{n,k}y_{\rho,\sigma}$ for $|\rho|=n+k, |\sigma|=k$. 
We distinguish the three cases $|\rho|=0,|\sigma|=0$ and $\min\{|\rho|,|\sigma|\}>0$. 
Since $P_{n,-n}=Q_{n,-n}$ we find that 
$x_{\emptyset,\sigma}=y_{\emptyset,\sigma}=N^{|\sigma|/2}W_{\emptyset, e_{\sigma}}$. 
For $|\sigma|=0$ we have that $Q_{n,k-1}W_{\rho,\emptyset}=0$ and thus in this case 
$x_{\rho, \sigma}=y_{\rho,\sigma}$ as well. 
The generic case follows from a straightforward application of Lemma \ref{lem:baby-step1}:
\begin{align*} Q_{n,k-1}N^{|\sigma|/2}W_{\rho,\sigma}
&=Q_{n,k-1}N^{|\sigma|/2}W_{e_\rho,e_\sigma}=N^{|\sigma|/2}W_{e_{\underline{\rho}}{}_{C(M)}(e_{\rho_{|\rho|}}|\mathfrak{q}e_{\sigma_{|\sigma|}}),e_{\underline{\sigma}}}\\
&=N^{|\sigma|/2}W_{e_{\underline{\rho}}\alpha^{-1}(N^{-1}e_{\sigma_{|\sigma|}}|e_{\rho_{|\rho|}}),e_{\underline{\sigma}}}=N^{|\sigma|/2-1}\alpha^{-|\rho|}(e_{\sigma_{|\sigma|}}|e_{\rho_{|\rho|}})W_{\underline{\rho},\underline{\sigma}}.
\end{align*}
where $|\rho|=n+k>0$. The formula for the frame now follows readily.
\end{proof}

Denote by $_{\alpha}\Xi_{C^{1}(M)}^{n,k}$ the $C^{1}(M)$-submodule 
of $_{\alpha}\Xi^{n,k}_{C(M)}$ 
generated by $x_{\rho,\sigma}$ as in 
Lemma \ref{globalframe} 
with $|\rho|=n+k$ and $|\sigma|=k$. 
The frame induces a connection $\nabla^{n,k}$ on each finite projective module 
 $_{\alpha}\Xi^{n,k}_{C^{1}(M)}$. 
The connections $\nabla^{n,k}$ allow us to define 
twisted Dirac operators $T_{n,k}:=1\otimes_{\nabla^{n,k}}\Dsla$ on 
$\aXi^{n,k}_{V}\otimes S$. We let $T$ denote the densely defined 
operator on $\mathcal{H}$ with domain
$$\Dom(T):=\left\{f=(f_{n,k})_{k,n}\in  \bigoplus_{k,n}L^2(M, \aXi_V^{n,k}\otimes S): 
\sum_{k,n} \|T_{n,k}f_{n,k}\|_{ L^2(M, \aXi_V^{n,k}\otimes S)}^2<\infty\right\},$$
defined by $T(f_{n,k})_{k,n}:=(T_{n,k}f_{n,k})_{k,n}$.

\begin{lemma}
\label{domainsandshit}
The operators $D$ and $T$ are self-adjoint and anti-commute with each other on 
their common core 
$$
X:=\left(\bigoplus^{\textnormal{alg}}_{n,k}\aXi^{n,k}_{C^{1}(M)}\right)
\otimes^{\textnormal{alg}}_{C^{1}(M)} \Dom\Dsla
$$
Moreover, $\D_E:=D+T$ is a self-adjoint operator $\D_E$ with compact resolvent.
\end{lemma}

\begin{proof} 
It is clear from their definitions that $X$ is a common core for $D$ and $T$.   
Both $T$ and $D$ respect the decomposition 
$\mathcal{H}=\bigoplus_{k,n} L^2(M, \aXi_V^{n,k}\otimes S)$ 
in the sense that $$D,T:\aXi^{n,k}_{C^{1}(M)}
\otimes^{\textnormal{alg}}_{C^{1}(M)} \Dom\Dsla\to\aXi^{n,k}_{C^{1}(M)}
\otimes^{\textnormal{alg}}_{C^{1}(M)} L^{2}(M,S).$$ 
In fact, $D$ maps $X$ into itself whereas $T$ maps $X$ into $\Dom D$. 
Therefore the anticommutator $DT+TD$ is defined on $X$ and is easily 
seen to vanish there. It then follows that  the sum $\D_E:=D+T$ is closed and $D+T$ is an essentially  
self-adjoint operator on the initial domain $X$, \cite[Theorem 6.1.8]{MeslandCrelle}. The resolvent of $\D_E^2$ can be written as 
$$
(1+ \D_E^2)^{-1}=\bigoplus_{k, n}\left(1+\psi(n,k)^2+T_{n,k}^2\right)^{-1}.
$$
For each $n,k$, $(1+\psi(n,k)^2+T_{n,k}^2)^{-1}$ is compact with 
$$\|(1+\psi(n,k)^2+T_{n,k}^2)^{-1}\|\leq (1+\psi(n,k)^2)^{-1}\to 0.$$
Therefore $(1+ \D_E^2)^{-1}$ is compact.
\end{proof}
In the sequel we will show that the commutators $[\D_{E},S_{\eta}\otimes 1]$ for $\eta\in\Gamma^{1}(M,V)$ are bounded 
assuming $\phi$ is compatible with the Riemannian metric (e.g. an isometry). From Lemma \ref{domainsandshit}, and by
checking the conditions of \cite[Theorem 13]{Kucerovsky}, we then deduce
the following theorem.

\begin{thm}
\label{kasppordvbmand}
Let $V\to M$ be a Hermitian vector bundle on a closed manifold, $\phi:M\to M$ an isometric $C^{1}$- diffeomorphism, 
$\D_E$ the operator constructed from a Dirac operator on $M$ as in Lemma \ref{domainsandshit} and 
$\mathcal{A}$ the dense $*$-subalgebra 
of $\O_{{}_{\alpha}E}$ generated by $S_{\eta}$ with $\eta \in \Gamma^1(M,V)$. 
The triple $(\A,\mathcal{H},\D_E)$ 
is a spectral triple for the Cuntz-Pimsner algebra $\O_{{}_{\alpha}E}$ representing the Kasparov 
product of the class of
$$
0\to \K_{C(M)}(\mathcal{F}_{{}_{\alpha}E})\to 
\T_{{}_{\alpha}E}\to 
\O_{{}_{\alpha}E}\to 0
$$
in $KK^1(\O_{{}_{\alpha}E},C(M))$ with $[\Dsla]\in KK^*(C(M),\C)$. The statements remain true 
if $\phi$ is a $C^1$ diffeomorphism such that for all $a\in \A$ there exists $C_a>0$ such that
$\sup_{\ell\in \Z}\Vert[\Dsla,\alpha^\ell(a)]\Vert\leq C_a$.
\end{thm}

\subsection{Proof of Theorem \ref{kasppordvbmand}}
We turn to the proof of Theorem \ref{kasppordvbmand} by proving boundedness of 
the commutators $[\D_{E},S_{\eta}\otimes 1]$ for $\eta\in\Gamma^{1}(M,V)$.
In the special case when the isometric diffeomorphism $\phi$ is the identity, boundedness 
of the commutators $[\D_{E},S_{\eta}\otimes 1]$ can be proved by a quick geometric argument. 
We prove the general case of a general isometric $C^{1}$-diffeomorphism 
$\phi:M\to M$ directly using the frame in Lemma \ref{globalframe}. 
To this end, we first establish some algebraic relations, describing the interaction of the algebra $C(M)$ and the operators $S_{e_{\iota}}$ with the global frame $x_{\rho,\sigma}$ constructed in Lemma \ref{globalframe}.

\begin{lemma}
\label{framecommute} 
For $a\in C(M)$ we have the identity
$ax_{\rho,\sigma}=x_{\rho,\sigma}\alpha^{|\rho|-|\sigma|}(a)$.
\end{lemma}

\begin{proof} 
The relation is obtained from the corresponding relations for $S_{\rho,\sigma}$ by writing
\[a W_{\rho,\sigma}=a S_{\rho,\sigma}W_{\emptyset,\emptyset}=
S_{\rho,\sigma}\alpha^{|\rho|-|\sigma|}(a)W_{\emptyset,\emptyset}=
S_{\rho,\sigma}W_{\emptyset,\emptyset}\alpha^{|\rho|-|\sigma|}(a)=
W_{\rho,\sigma}\alpha^{|\rho|-|\sigma|}(a),\]
and then using that $|\rho|-|\sigma|=|\underline{\rho}|-|\underline{\sigma}|$ so that the relation 
passes to the $x_{\rho,\sigma}$ in all cases.
\end{proof}
 
\begin{lemma}
\label{Srelations}
For $|\iota|=1$ we have the relations:
\[S_{e_{\iota}}x_{\rho,\sigma}=
\left\{\begin{matrix} x_{\iota,\sigma}+N^{-1/2}\alpha^{-1}(e_{\sigma_{|\sigma|}}|e_{\iota})x_{\emptyset,\underline{\sigma}} & |\rho|=0 \\ x_{\iota\rho,\sigma}& |\rho|>0\end{matrix}\right. \quad S_{e_{\iota}}^{*}x_{\rho,\sigma}=\left\{\begin{matrix} N^{-1/2}x_{\emptyset,\sigma\iota} & |\rho|=0\\
(e_{\iota}|e_{\rho})x_{\emptyset,\sigma}-N^{-1}(e_{\sigma_{|\sigma|}}|e_{\rho})x_{\emptyset,\underline{\sigma}\iota} & |\rho|=1 \\
(e_{\iota}| e_{\rho_{1}}) x_{\overline{\rho},\sigma} & |\rho|>1\end{matrix}\right.\]
with the convention that $e_{\emptyset_{|\emptyset|}}=0$.
\end{lemma}
\begin{proof} For the operator $S_{e_\iota}$, the action on $x_{\rho,\sigma}$ for $|\rho|>0$ is straightforward to check. 
For $|\rho|=0$ we compute
\begin{align*}
S_{e_{\iota}}x_{\emptyset,\sigma}&=S_{e_{\iota}}N^{|\sigma|/2}W_{\emptyset,\sigma}
=N^{|\sigma|/2}W_{\iota,\sigma}\\ 
&=N^{|\sigma|/2}W_{\iota,\sigma}\!-N^{|\sigma|/2-1}\alpha^{-1}(e_{\sigma_{|\sigma|}}|e_{\iota})W_{\emptyset,\underline{\sigma}} 
+N^{|\sigma|/2-1}(e_{\sigma_{|\sigma|}}|e_{\iota})W_{\emptyset,\underline{\sigma}}\!=
x_{\iota,\sigma}\!+N^{-1/2}\alpha^{-1}(e_{\sigma_{|\sigma|}}|e_{\iota})x_{\emptyset,\underline{\sigma}}.
\end{align*}
For $S_{e_{\iota}}^{*}$, the relations for $|\rho|=0$ and $|\rho|>1$ are straightforward to check. For $|\rho|=1$ we compute
\begin{align*} 
S_{e_{\iota}}^{*}x_{\rho,\sigma}&=S_{e_{\iota}}^{*}(N^{|\sigma|/2}W_{\rho,\sigma}-N^{|\sigma|/2-1}(e_{\sigma_{|\sigma|}}|e_{\rho})W_{\emptyset,\underline{\sigma}})\\
&=N^{|\sigma|/2}(e_{\iota}|e_{\rho})W_{\emptyset,\sigma}-N^{|\sigma|/2-1}(e_{\sigma_{|\sigma|}}|e_{\rho})W_{\emptyset,\underline{\sigma}\iota}\\
&=(e_{\iota}|e_{\rho})x_{\emptyset,\sigma}-N^{-1}(e_{\sigma_{|\sigma|}}|e_{\rho})x_{\emptyset,\underline{\sigma}\iota}, 
\end{align*}
as claimed.
\end{proof}

\begin{lemma}
\label{crackerbewdycorkerripperbonzagrouse} 
The following relations hold:\newline\newline
1.) $\sum_{|\lambda|=1}x_{\lambda\rho,\sigma}\alpha^{|\rho|-|\sigma|}(e_{\lambda}|e_{\iota})
=x_{\iota\rho,\sigma}$ for all $\rho$;\newline
2.) $\sum_{|\lambda|=1} x_{\emptyset,\sigma\lambda}\alpha^{-|\sigma|-1}(e_{\iota} | e_{\lambda})
=x_{\emptyset,\sigma\iota}$;\newline
3.) $\sum_{|\lambda|=1} x_{\lambda,\sigma\lambda}=0$.
\end{lemma}

\begin{proof} 
The identities all rely on the frame relation. For 1.) and $|\rho|>0$ :
\begin{align*}\sum_{|\lambda|=1} x_{\lambda\rho,\sigma}\alpha^{|\rho|-|\sigma|}(e_{\lambda} | e_{\iota})&= 
\sum_{|\lambda|=1} N^{|\sigma|/2}W_{\lambda\rho,\sigma}\alpha^{|\rho|-|\sigma|}(e_{\lambda} | e_{\iota})-N^{|\sigma|/2-1}\alpha^{-|\rho|-1}(e_{\sigma_{|\sigma|}}|e_{\lambda})W_{\lambda\underline{\rho},\underline{\sigma}}\alpha^{|\underline{\rho}|-|\underline{\sigma}|}(e_{\lambda} | e_{\iota})\\
&= N^{|\sigma|/2}W_{\iota\rho,\sigma}-N^{|\sigma|/2-1}\alpha^{-|\rho|-1}(e_{\sigma_{|\sigma|}}|e_{\lambda})W_{\iota\underline{\rho},\underline{\sigma}}=x_{i\rho,\sigma}
\end{align*}
and for $|\rho|=0$:
\begin{align*}\sum_{|\lambda|=1} x_{\lambda,\sigma}\alpha^{|\rho|-|\sigma|}(e_{\lambda} | e_{\iota}) & = 
\sum_{|\lambda|=1} N^{|\sigma|/2}W_{\lambda,\sigma}\alpha^{-|\sigma|}(e_{\lambda} | e_{\iota})-N^{|\sigma|/2-1}\alpha^{-1}(e_{\sigma_{|\sigma|}}|e_{\lambda})W_{\emptyset,\underline{\sigma}}\alpha^{-|\sigma|}(e_{\lambda} | e_{\iota})\\
&=N^{|\sigma|/2}W_{\iota,\sigma}-\sum_{|\lambda|=1}N^{|\sigma|/2-1}\alpha^{-1}((e_{\sigma_{|\sigma|}}|e_{\lambda})(e_{\lambda}|e_{\iota}))W_{\emptyset,\underline{\sigma}}\\
&=N^{|\sigma|/2}W_{\iota,\sigma}-N^{|\sigma|/2-1}\alpha^{-1}((e_{\sigma_{|\sigma|}}|e_{\iota}))W_{\emptyset,\underline{\sigma}}=x_{\iota,\sigma}.
\end{align*}
Identity 2.) relies on similar considerations, observing that
\[\sum_{|\lambda|=1} x_{\emptyset,\sigma\lambda}\alpha^{-|\sigma|-1}(e_{\iota}|e_{\lambda})=\sum_{|\lambda|=1} x_{\emptyset,\sigma\lambda(e_{\lambda}|e_{\iota})}=x_{\emptyset,\sigma\iota}.\]
For 3.) we also use $\alpha$-invariance of the Jones-Watatani index:
\begin{align*}\sum_{|\lambda|=1} x_{\lambda,\sigma\lambda}&=\sum_{|\lambda|=1}N^{(|\sigma|+1)/2}W_{\lambda,\sigma\lambda}-N^{(|\sigma|-1)/2}\alpha^{-1}(e_\lambda|e_\lambda)W_{\emptyset,\sigma}\\
&=N^{(|\sigma|+1)/2}(\sum_{|\lambda|=1}W_{\lambda,\sigma\lambda})-N^{(|\sigma|-1)/2}\alpha^{-1}(\sum_{|\lambda|=1}(e_\lambda|e_\lambda))W_{\emptyset,\sigma}\\&
=N^{(|\sigma|+1)/2}W_{\emptyset,\sigma}-N^{(|\sigma|-1)/2}N W_{\emptyset,\sigma}
=0.\qedhere
\end{align*}
\end{proof}

For $C^{1}(M,V)\subset E$, the $C^{1}(M)$-submodule of $C^{1}$-sections of $V$, the tensor products $C^{1}(M,V)^{\otimes \ell}$ are understood to be algebraic tensor products balanced over the action of $C^{1}(M)$ through $\alpha$. It is then automatic that for $f,g\in C^{1}(M,V)^{\otimes \ell}$ it holds that $(f|g)\in C^{1}(M)$.

\begin{lemma}
\label{forgettrivialbundles}
For $|\iota|=1$, $\mu\in C^{1}(M,V)^{\otimes |\mu|},\nu\in C^{1}(M,V)^{\otimes |\nu|}$ and $\xi\in \Dom \Dsla$, we have the idenitity
\begin{align}[T,&S_{e_{\iota}}\otimes 1] W_{\mu,\nu}\otimes \xi \label{commutator}=\\ &=\sum_{|\lambda|=1}\sum_{|\rho|>0,\sigma} x_{\lambda\rho,\sigma}\otimes[\Dsla,\alpha^{|\rho|-|\sigma|}(e_{\lambda}|e_{\iota})](x_{\rho,\sigma}|W_{\mu,\nu})\xi
+\sum_{\sigma}N^{-1/2}x_{\emptyset,\underline{\sigma}}\otimes [\Dsla,\alpha^{-|\sigma|}(e_{\sigma_{|\sigma|}}|e_{\iota})] (x_{\emptyset,\sigma} | W_{\mu,\nu})\xi\nonumber\\
&\quad+\sum_{|\lambda|=1}\sum_{\sigma}x_{\lambda,\sigma}\otimes [\Dsla,\alpha^{-|\sigma|}(e_{\lambda}|e_{\iota})] (x_{\emptyset,\sigma} | W_{\mu\nu})\xi - N^{-1} x_{\lambda,\sigma}\otimes[\Dsla,\alpha^{-|\sigma|}(e_{\lambda}|e_{\sigma_{|\sigma|}})] (  x_{\emptyset,\underline{\sigma}\iota} | W_{\mu,\nu})\xi\nonumber
\end{align}
\end{lemma}

\begin{proof} 
We let $[T,S_{e_{\iota}}\otimes1]$ act on $W_{\mu,\nu}$ and compute:
\begin{align}\nonumber [T,S_{e_{\iota}}\otimes1] W_{\mu,\nu}\otimes \xi &= 
\sum_{\rho,\sigma} x_{\rho,\sigma}\otimes \Dsla(S_{e_{\iota}}^{*}x_{\rho,\sigma} | W_{\mu,\nu}) \xi - S_{e_{\iota}}x_{\rho,\sigma}\otimes\Dsla(x_{\rho,\sigma}| W_{\mu,\nu})\xi\\
&\label{decentpart1}
=\sum_{|\rho|>1,\sigma} x_{\rho,\sigma}\otimes\Dsla(S_{e_{\iota}}^{*}x_{\rho,\sigma}|W_{\mu,\nu})\otimes\xi\\
&\quad
\label{datcrapagain}
+\sum_{|\lambda|=1}\sum_{\sigma}x_{\lambda,\sigma}\otimes\Dsla(S_{e_{\iota}}^{*}x_{\lambda,\sigma} | W_{\mu,\nu}) \xi\\
&\quad
\label{buthereisalwaysmore}
+\sum_{\sigma}x_{\emptyset,\sigma}\otimes\Dsla(S_{e_{\iota}}^{*}x_{\emptyset,\sigma}|W_{\mu,\nu})\xi\\
&\label{decentpart2}\quad 
-\sum_{|\rho|>0,\sigma} S_{e_{\iota}}x_{\rho,\sigma}\otimes \Dsla(x_{\rho,\sigma}|W_{\mu,\nu})\xi\\
&\quad \label{youthoughtyoudseenallcrap}
-\sum_{\sigma} S_{e_\iota}x_{\emptyset,\sigma}\otimes \Dsla(x_{\emptyset,\sigma}|W_{\mu,\nu})\xi
\end{align}
We proceed with \eqref{decentpart1}, using Lemmas \ref{Srelations} and \ref{crackerbewdycorkerripperbonzagrouse} 1.) with $n=|\rho|-|\sigma|$ :
\begin{align*}\eqref{decentpart1}&=\sum_{|\lambda|=1}\sum_{|\rho|>0,\sigma} x_{\lambda\rho,\sigma}\otimes\Dsla((e_{\iota}|e_{\lambda})x_{\rho,\sigma}|W_{\mu,\nu})\xi\\ 
&= \sum_{|\lambda|=1}\sum_{|\rho|>0,\sigma} x_{\lambda\rho,\sigma}\otimes[\Dsla,\alpha^{n}(e_{\lambda}|e_{\iota})](x_{\rho,\sigma}|W_{\mu,\nu})\xi + \sum_{|\lambda|=1}\sum_{|\rho|>0,\sigma} x_{\lambda\rho,\sigma}\alpha^{n}(e_{\lambda}|e_{\iota})\otimes  \Dsla(x_{\rho ,\sigma} | W_{\mu,\nu})\xi\\
&=\sum_{|\lambda|=1}\sum_{|\rho|>0,\sigma} x_{\lambda\rho,\sigma}\otimes[\Dsla,\alpha^{n}(e_{\lambda}|e_{\iota})](x_{\rho,\sigma}|W_{\mu,\nu})\xi +\sum_{|\rho|>0,\sigma} x_{\iota\rho,\sigma}\otimes\Dsla (x_{\rho ,\sigma} | W_{\mu,\nu})\xi\\
&=\sum_{|\lambda|=1}\sum_{|\rho|>0,\sigma} x_{\lambda\rho,\sigma}\otimes[\Dsla,\alpha^{n}(e_{\lambda}|e_{\iota})](x_{\rho,\sigma}|W_{\mu,\nu})\xi -\eqref{decentpart2},
\end{align*}
from which we see that \eqref{decentpart1} and \eqref{decentpart2} add up to
\begin{equation}
\label{bddterm1}
\sum_{|\lambda|=1}\sum_{|\rho|>0,\sigma} x_{\lambda\rho,\sigma}\otimes
[\Dsla,\alpha^{|\rho|-|\sigma|}(e_{\lambda}|e_{\iota})](x_{\rho,\sigma}|W_{\mu,\nu})\xi.
\end{equation} 
We proceed with \eqref{youthoughtyoudseenallcrap} using Lemmas \ref{framecommute}, \ref{Srelations} and \ref{crackerbewdycorkerripperbonzagrouse} 2.):
\begin{align}
\sum_{\sigma} -S_{e_\iota} x_{\emptyset,\sigma}\otimes \Dsla(x_{\emptyset,\sigma}|W_{\mu,\nu})\xi &=
\sum_{\sigma} -x_{\iota,\sigma}\otimes \Dsla(x_{\emptyset,\sigma} | W_{\mu,\nu})\xi-N^{-1/2}\alpha^{-1}(e_{\sigma_{|\sigma|}}|e_{\iota})x_{\emptyset,\underline{\sigma}}\otimes\Dsla (x_{\emptyset,\sigma} | W_{\mu,\nu})\xi\nonumber \\
&=\sum_{\sigma} -x_{\iota,\sigma}\otimes \Dsla(x_{\emptyset,\sigma} | W_{\mu,\nu})\xi-N^{-1/2}x_{\emptyset,\underline{\sigma}}\otimes\Dsla (x_{\emptyset,\sigma}\alpha^{-|\sigma|}(e_{\iota}|e_{\sigma_{|\sigma|}}) | W_{\mu,\nu})\xi\nonumber\\ 
&\quad+N^{-1/2}x_{\emptyset,\underline{\sigma}}\otimes [\Dsla,\alpha^{-|\sigma|}(e_{\sigma_{|\sigma|}}|e_{\iota})] (x_{\emptyset,\sigma} | W_{\mu,\nu})\xi\nonumber \\
&=\sum_{\sigma} -x_{\iota,\sigma}\otimes \Dsla(x_{\emptyset,\sigma} | W_{\mu,\nu})\xi-N^{-1/2}x_{\emptyset,\sigma}\otimes\Dsla (x_{\emptyset,\sigma\iota} | W_{\mu,\nu}) \xi
\label{letsgetreadytorumble}\\
&\quad+N^{-1/2}x_{\emptyset,\underline{\sigma}}\otimes [\Dsla,\alpha^{-|\sigma|}(e_{\sigma_{|\sigma|}}|e_{\iota})] (x_{\emptyset,\sigma} | W_{\mu,\nu})\xi
\label{bddterm2}.
\end{align}
Next, we turn to \eqref{datcrapagain} again applying Lemma \ref{Srelations}:
\begin{align}
\eqref{datcrapagain} 
&=\sum_{|\lambda|=1}\sum_{\sigma}x_{\lambda,\sigma}\otimes \Dsla((e_{\iota}|e_{\lambda})x_{\emptyset,\sigma} | W_{\mu\nu})\xi -\sum_{|\lambda|=1}\sum_{\sigma} N^{-1} x_{\lambda,\sigma}\otimes \Dsla( (e_{\sigma_{|\sigma|}}| e_{\lambda}) x_{\emptyset,\underline{\sigma}\iota} | W_{\mu,\nu})\xi
\nonumber \\
&=\sum_{|\lambda|=1}\sum_{\sigma}x_{\lambda,\sigma}\alpha^{-|\sigma|}(e_{\lambda}|e_{\iota})\otimes\Dsla (x_{\emptyset,\sigma} | W_{\mu\nu})\xi -\sum_{|\lambda|=1}\sum_{\sigma} N^{-1} x_{\lambda,\sigma}\alpha^{-|\sigma|}(e_{\lambda}|e_{\sigma_{|\sigma|}})\otimes\Dsla (  x_{\emptyset,\underline{\sigma}\iota} | W_{\mu,\nu})\xi
\label{roadto}\\
&+\sum_{|\lambda|=1}\sum_{\sigma}x_{\lambda,\sigma}\otimes [\Dsla,\alpha^{-|\sigma|}(e_{\lambda}|e_{\iota})] (x_{\emptyset,\sigma} | W_{\mu\nu})\xi -\sum_{|\lambda|=1}\sum_{\sigma} N^{-1} x_{\lambda,\sigma}\otimes[\Dsla,\alpha^{-|\sigma|}(e_{\lambda}|e_{\sigma_{|\sigma|}})] (  x_{\emptyset,\underline{\sigma}\iota} | W_{\mu,\nu})\xi
\label{bddterm3}.
\end{align}
Considering \eqref{roadto} and applying Lemma  \ref{crackerbewdycorkerripperbonzagrouse} 1.), 2.)  and 3.) we find
\begin{align}
\eqref{roadto} 
=\sum_{\sigma}x_{\iota,\sigma}\otimes\Dsla (x_{\emptyset,\sigma} | W_{\mu\nu})\xi -\sum_{\sigma} \sum_{|\lambda|=1}N^{-1} x_{\lambda,\sigma\lambda}\otimes \Dsla(  x_{\emptyset,\sigma\iota} | W_{\mu,\nu})\xi=\sum_{\sigma}x_{\iota,\sigma}\otimes \Dsla(x_{\emptyset,\sigma} | W_{\mu\nu})\xi. 
\label{salvation}
\end{align}
Lastly, we compute \eqref{buthereisalwaysmore}
\begin{align}
\eqref{buthereisalwaysmore} 
=\sum_{\sigma}x_{\emptyset,\sigma}\otimes \Dsla (S_{e_{\iota}}^{*}x_{\emptyset,\sigma}|W_{\mu,\nu})\xi
=\sum N^{-1/2} x_{\emptyset,\sigma} \otimes\Dsla (x_{\emptyset,\sigma\iota} | W_{\mu,\nu}) \xi.
\label{jesusisreal}
\end{align}
Now we see that \eqref{letsgetreadytorumble},\eqref{salvation} and \eqref{jesusisreal} add up to $0$. 
Thus we are left with \eqref{bddterm1},\eqref{bddterm2} and \eqref{bddterm3}, 
which together yield the expression \eqref{commutator}.
\end{proof}

To ensure the bounded commutators in Theorem \ref{kasppordvbmand}, 
we impose a further condition on $\alpha$ (cf. \cite{BMR})
which is automatically satisfied if $\phi$ is an isometry. 

\begin{prop} 
\label{prop:better-diffeo} 
Let $\phi:M\to M$ be a $C^{1}$-diffeomorphism  
whose associated automorphism $\alpha$ is such that for all $a\in C^{1}(M)$ there exists a $C_a>0$ for which 
$\sup_{\ell\in \Z}\|[\Dsla, \alpha^{\ell}(a)]\|\leq C_{a}$. 
For $|\iota|=1$ the operator $S_{e_\iota}\otimes 1$ maps the core $X$ 
described in Lemma \ref{domainsandshit} into $\Dom \D_{E}$ and 
the commutator $[\D_{E},S_{e_{\iota}}\otimes 1]$ extends to a bounded operator.
\end{prop}

\begin{proof} 
It is clear that $S_{e_{\iota}}$ maps $X$ into itself. 
For the commutator, observe that 
$$[\mathcal{D}_{E},S_{e_\iota}\otimes 1]=[D,S_{e_\iota}\otimes 1]+[T,S_{e_\iota}\otimes 1],$$ 
and $[D,S_{e_\iota}\otimes 1]$ is bounded by construction. For 
$[T,S_{e_\iota}\otimes 1]$ we use Lemma \ref{forgettrivialbundles} 
and analyse the four terms in Equation \eqref{commutator}. 
All terms can be shown to be bounded by a similar method. For instance, consider 
\begin{equation}
\label{part} 
W_{\mu,\nu}\otimes \xi\mapsto \sum_{|\lambda|=1}\sum_{|\rho|>0,\sigma} 
x_{\lambda\rho,\sigma}\otimes[\Dsla,\alpha^{|\rho|-|\sigma|}(e_{\lambda}|e_{\iota})]
(x_{\rho,\sigma}|W_{\mu,\nu})\xi.
\end{equation}
Consider the partial isometry
\[
V:\left(\bigoplus_{k, n}{}_{\alpha}\Xi_{C(M)}^{n,k}\right)\otimes_{C(M)}L^{2}(M,S)
\to \bigoplus_{|\rho|>0,\sigma}L^{2}(M,S),\quad 
W_{\mu,\nu}\otimes \xi\to ( (x_{\rho,\sigma}|W_{\mu,\nu})\xi)_{\rho,\sigma},
\]
and the map $M_{\iota}$ defined through $M_{\iota}:=\diag_{n,k}(M_{\iota}^{n,k})$, where 
\[ 
M_{\iota}^{n,k}:\bigoplus_{|\rho|=n+k, |\sigma|=k}L^{2}(M,S)
\to \bigoplus_{|\rho|=n+k+1,|\sigma|=k} L^{2}(M,S),\quad (M_{\iota}^{n,k}\xi)_{\rho,\sigma}
:=([\Dsla,\alpha^{|\rho|-|\sigma|-1}(e_{\rho_{1}}|e_{\iota})])\xi_{\overline{\rho},\sigma}.
\] The operator in Equation \eqref{part} then coincides with the composition 
$V^{*}M_{\iota}V$. It thus suffices to show that 
$\sup_{n,k}\|M^{n,k}_{\iota}\|<\infty$, which follows from 
the assumption that $\sup_{\ell\in \Z}\|[\Dsla, \alpha^{\ell}(a)]\|\leq C_{a}$ 
for each $a$ and the fact that the frame $e_{\lambda}$ has 
finitely many elements. Thus \eqref{part} defines a bounded operator. 
The other summands in Equation \eqref{commutator} can be 
shown to be bounded by a similar argument.
\end{proof}

\end{document}